\documentclass[11pt]{article}

\usepackage{amsmath,amsthm,verbatim,amssymb,amsfonts,amscd, graphicx}
\usepackage{graphics}
\usepackage{authblk}
\usepackage{upgreek}
\usepackage{amsrefs}
\theoremstyle{plain}
\newtheorem{theorem}{Theorem}[section]

\newtheorem{corollary}[theorem]{Corollary}
\newtheorem{definition}[theorem]{Definition}

\newtheorem{proposition}[theorem]{Proposition}
\theoremstyle{remark}
\newtheorem{remark}[theorem]{Remark}
\numberwithin{equation}{section}
\newcommand{\diff}{\mathop{}\!\mathrm{d}}
\DeclareMathOperator*{\Hess}{Hess}
\DeclareMathOperator*{\tr}{tr}
\DeclareMathOperator*{\divr}{div}

\title{Finite-time blowup for a Navier--Stokes model equation for the self-amplification of strain}
\author[1]{Evan Miller}
\affil[1]{University of British Columbia, Department of Mathematics

emiller@msri.org}

\begin{document}

\maketitle

\begin{abstract}
In this paper, we consider a model equation for the Navier--Stokes strain equation. This model equation has the same identity for enstrophy growth and a number of the same regularity criteria as the full Navier--Stokes strain equation, and is also an evolution equation on the same constraint space.
We prove finite-time blowup for this model equation, which shows that the identity for enstrophy growth and the strain constraint space are not sufficient on their own to guarantee global regularity for Navier--Stokes.
The mechanism for the finite-time blowup of this model equation is the self-amplification of strain, which is consistent with recent research suggesting that strain self-amplification, not vortex stretching, is the main mechanism behind the turbulent energy cascade.
Because the strain self-amplification model equation is obtained by dropping certain terms from the full Navier--Stokes strain equation, we will also prove a conditional blowup result for the full Navier--Stokes equation involving a perturbative condition on the terms neglected in the model equation.
\end{abstract}

\section{Introduction}

The incompressible Navier--Stokes equation is one of the fundamental equations of fluid mechanics. Although it is over 150 years old, much about its solutions, including the global existence of smooth solutions, remains unknown. The Navier--Stokes equation is given by
\begin{align} \label{NavierStokes}
    \partial_t u-\Delta u 
    +(u \cdot \nabla)u+\nabla p&=0\\
    \nabla \cdot u&=0,
\end{align}
where $u\in\mathbb{R}^3$ is the velocity and $p$ is the pressure. The first equation is a statement of Newton's second law, $F=ma$, where $\partial_t u+(u\cdot\nabla)u$
gives the acceleration in the Lagrangian frame, $\Delta u$ describes the viscous forces due to the internal friction of the fluid, and $-\nabla p$ describes the force due to the pressure.
The second equation, the divergence free constraint, comes from the conservation of mass.
We will note that $p$ is not an independently evolving function, but is determined entirely by $u$ by convolution with the Poisson kernel,
\begin{equation}
    p=(-\Delta)^{-1}\sum_{i,j=1}^3
    \frac{\partial u_j}{\partial x_i}
    \frac{\partial u_i}{\partial x_j}.
\end{equation}

It is possible to state the incompressible Navier--Stokes equation without giving any reference to pressure at all by making use of the Helmholtz projection onto the space of divergence free vector fields, yielding the equation
\begin{equation} \label{NSprojection}
    \partial_t u -\Delta u 
    +P_{df}\nabla \cdot \left(u\otimes u\right)=0.
\end{equation}
Note that we have used the fact that 
$\nabla \cdot (u\otimes u)=(u\cdot\nabla)u,$ 
because $\nabla \cdot u=0$, and the fact that
the Helmholtz decomposition implies that
$P_{df}(\nabla p)=0$.

The first major advances towards a rigorous mathematical understanding of the Navier--Stokes equation came in the seminal paper by Leray \cite{Leray}. For all initial data $u^0 \in L^2_{df},$ Leray proved the global-in-time existence of weak solutions, in the sense of integrating against smooth test functions, satisfying the energy inequality, which states that for all $t>0$,
\begin{equation} \label{EnergyInequality}
    \frac{1}{2}\|u(t)\|_{L^2}^2+
    \int_{0}^t \|u(\tau)\|_{\dot{H}^1}^2 \diff\tau
    \leq \frac{1}{2}\left\|u^0\right\|_{L^2}^2.
\end{equation}
Unfortunately, while such solutions are well suited to study in that the sense that global-in-time existence is guaranteed for all finite-energy initial data, they are not known to be either smooth or unique, leaving major problems for the well-posedness theory.

The lack of a uniqueness and regularity theory for Leray weak solutions led Fujita and Kato to develop the notion of mild solutions,
which satisfy \eqref{NSprojection} in the sense of Duhamel's formula.
Unlike Leray's weak solutions, mild solutions must be both smooth and unique.
Kato and Fujita proved the local-in-time existence, uniqueness, and smoothness of mild solutions for initial data in $\dot{H}^1$, with the time of existence bounded below uniformly in the $\dot{H}^1$ norm \cite{KatoFujita}.

\begin{theorem} \label{KatoMildExistence}
There exists an absolute constant $C>0,$ such that for all initial data $u^0\in \dot{H}^1_{df}$, there exists $T_{max}\geq \frac{C}{\left\|u^0\right\|_{\dot{H}^1}^4}$ and a unique mild solution to the Navier--Stokes equation
$u\in C\left([0,T_{max});\dot{H}^1_{df}\right)$. Furthermore, we have the higher regularity
$u\in C^\infty\left((0,T_{max})\times \mathbb{R}^3\right).$ If in addition we have $u^0\in H^1_{df},$ then the energy inequality holds with equality, that is for all $0<t<T_{max}$,
\begin{equation}
    \frac{1}{2}\|u(t)\|_{L^2}^2+
    \int_{0}^t \|u(\tau)\|_{\dot{H}^1}^2 \diff\tau
    = \frac{1}{2}\left\|u^0\right\|_{L^2}^2.
\end{equation}
\end{theorem}

We will note that because mild solutions are smooth and unique, the initial value problem for mild solutions of the Navier--Stokes equation is locally well-posed in $\dot{H}^1$---and also in a number of larger spaces; however, it is not known to be globally well-posed. Whether the Navier--Stokes equation has global smooth solutions or admits smooth solutions that blowup in finite-time is one of the biggest open problems in PDEs and one of the ``Millennium Problems'' put forward by the Clay Mathematics Institute \cite{Clay}.

The main difficulty is that the only bounds that are available on the growth of solutions are the bounds in 
$L^\infty_t L^2_x$ and $L^2_t \dot{H}^1_x$
due to the energy equality, and these bounds are not enough to guarantee the global existence of smooth solutions because the energy equality is super-critical with respect to the invariant rescaling of the Navier--Stokes equation. The solution set of the Navier--Stokes equation is preserved under the rescaling,
\begin{equation}
    u^{\lambda}(x,t)=\lambda u(\lambda x, \lambda^2 t),
\end{equation}
for all $\lambda>0.$
This means that is not enough to control the $L^\infty_t L^2_x$ or $L^2_t \dot{H}^1_x$ norms of $u,$ which are supercritical in terms of scaling; in order to guarantee global regularity, we need to control a scale critical norm.
Ladyzhenskaya \cite{Ladyzhenskaya}, Prodi \cite{Prodi}, and Serrin \cite{Serrin} independently proved a family of scale critical regularity criteria, which state that if $T_{max}<+\infty,$ and $\frac{2}{p}+\frac{3}{q}=1,$ with $3<q\leq +\infty$ then
\begin{equation}
    \int_0^{T_{max}}\|u(t)\|_{L^q}^p \diff t=+\infty.
\end{equation}
Escauriaza, Seregin and S\v{v}er\'ak \cite{L3} extended this result to the endpoint case $q=3$. They proved that if $T_{max}<+\infty,$ then
\begin{equation}
    \limsup_{t \to T_{max}}\|u(t)\|_{L^3}=+\infty.
\end{equation}
Recently, Tao further extended this regularity criterion giving a quantitative lower bound on the rate of the blowup of the $L^3$ norm \cite{TaoL3}. This result is very slightly supercritical---in fact triple logarithmically---with respect to scaling, and is the first supercritical regularity criterion for the Navier--Stokes equation.

Two crucially important objects for the study of the Navier--Stokes equation are the strain, which is the symmetric gradient of the velocity, $S=\nabla_{sym}u,$ with $S_{ij}=\frac{1}{2}\left(\partial_i u_j+ \partial_j u_i\right),$
and the vorticity, which is a vector that represents the anti-symmetric part of the velocity, and is given by $\omega=\nabla \times u.$
Physically, the strain describes how a parcel of the fluid is deformed, while the vorticity describes how a parcel of the fluid is rotated.

Taking the curl of \eqref{NavierStokes}, we find the evolution equation for $\omega$ is given by
\begin{equation}
    \partial_t\omega-\Delta\omega+(u\cdot\nabla)\omega-S\omega=0.
\end{equation}
Taking the symmetric gradient of \eqref{NavierStokes}, we find the evolution equation for $S$ is given by,
\begin{equation} \label{straineq}
    \partial_t S -\Delta S +(u\cdot \nabla)S
    +S^2+\frac{1}{4}\omega\otimes\omega
    -\frac{1}{4}|\omega|^2I_3+\Hess(p)=0.
\end{equation}
We will note that the vorticity equation is invariant under the rescaling,
\begin{equation}
    \omega^\lambda(x,t)=\lambda^2\omega(\lambda x,\lambda^2 t),
\end{equation}
and the strain equation is invariant under the rescaling
\begin{equation}
    S^\lambda(x,t)=\lambda^2 S(\lambda x,\lambda^2 t).
\end{equation}
The extra factor of $\lambda$ comes from the fact that both $\omega$ and $S$ scale like $\nabla u.$

The vorticity has been studied fairly exhaustively for its role in the dynamics of the Navier--Stokes equation.
For instance the Beale-Kato-Majda regularity criterion
\cite{BKM}, which holds for smooth solutions of both the Euler and Navier--Stokes equations, states that if $T_{max}<+\infty,$ then
\begin{equation}
    \int_0^{T_{max}}\|\omega(\cdot,t)\|_{L^\infty}\diff t
    =+\infty.
\end{equation}
Chae and Choe proved a regularity criterion on two components of vorticity that has a geometric significance, guaranteeing that blowup must be fully three dimensional \cite{ChaeVort}. They showed that if a smooth solution of the Navier--Stokes equation blows up in finite-time $T_{max}<+\infty$, then for all 
$\frac{3}{2}<q<+\infty, \frac{2}{p}+\frac{3}{q}=2$,
\begin{equation} \label{2vort}
    \int_{0}^{T_{max}} 
    \|e_3\times\omega(\cdot,t)\|_{L^q}^p \diff t
    =+\infty.
\end{equation}
The fixed direction condition in this regularity criterion was recently loosened by the author in \cite{MillerVort}.
In another key result involving vorticity, Constantin and Fefferman proved that the direction of the vorticity must vary rapidly in regions where the vorticity is large if there is finite-time blowup \cite{VortDirection}. 
There are many other results involving vorticity, far too many to list here.

The strain equation has been investigated much less thoroughly, but can provide some insights that do not follow as clearly from the vorticity equation.
We will refer to the evolution equation for $S$ in \eqref{straineq} as the Navier--Stokes strain equation. This equation is an evolution equation on the constraint space $L^2_{st},$ the space of strain matrices, which replaces the divergence free constraint for the Navier--Stokes and vorticity equations. We define $L^2_{st}$ as follows.

\begin{definition} \label{L2Strain}
Define $L^2_{st} \subset 
L^2\left(\mathbb{R}^3;\mathbb{S}^{3\times 3}\right)$ by
\begin{equation}
    L^2_{st}=\left\{ \nabla_{sym}u: u\in\dot{H}^1,
    \nabla\cdot u=0 \right\}.
\end{equation}
\end{definition}

The role of this constraint space in the evolution equation \eqref{straineq} was examined by the author in \cite{MillerStrain}. One geometric restriction on the matrices $S\in L^2_{st}$ is that they must be trace free, because 
\begin{equation}
\tr(S)=\nabla\cdot u =0.     
\end{equation}
Furthermore, in that paper, the author proved that Hessians and scalar multiples of the identity matrix must be in the orthogonal compliment of $L^2_{st}.$

\begin{proposition}
For all $f\in \dot{H}^2\left(\mathbb{R}^3\right),$ 
and for all $g \in L^2\left(\mathbb{R}^3\right),$ we have
$\Hess(f),g I_3\in \left(L^2_{st}\right)^\perp$.
That is for all $S \in L^2_{st},$
\begin{align}
    \left<\Hess(f),S\right>&=0,\\
    \left<g I_3,S\right>&=0.
\end{align}
\end{proposition}

For sufficiently smooth solutions to the Navier--Stokes strain equation $\frac{1}{4}|\omega|^2,\Hess(p)\in L^2,$ so we can conclude that the terms
$\frac{1}{4}|\omega|^2 I_3$ and $\Hess(p)$ are orthogonal to the constraint space, $\frac{1}{4}|\omega|^2 I_3, \Hess(p) \in \left(L^2_{st}\right)^\perp.$ This means that the Navier--Stokes strain equation can be expressed in terms of the projection onto $L^2_{st}$ as
\begin{equation} \label{NSstrain}
    \partial_t S -\Delta S +P_{st}\left((u\cdot \nabla)S
    +S^2+\frac{1}{4}\omega\otimes\omega \right)=0.
\end{equation}
This is analogous to defining the Navier--Stokes equation without any reference to $\nabla p$ by using the Helmholtz projection onto the space of divergence free vector fields in \eqref{NSprojection}.
We will use \eqref{NSstrain} to define mild solutions to the Navier--Stokes strain equation in section 3.

It is not actually necessary to separately prove the existence of mild solutions to the strain equation, as it is straightforward to reduce this problem to the existence of mild solutions of the Navier--Stokes equation. The author proved the equivalence of these formulations in \cite{MillerStrain}.

\begin{proposition} \label{MildEquivalence}
A velocity field $u\in C\left([0,T_{max});\dot{H}^1_{df}\right)$ is a mild solution of the Navier--Stokes equation, if and only if $S\in C\left([0,T_{max});L^2_{st}\right)$ is a mild solution to the Navier--Stokes strain equation, where 
$S=\nabla_{sym} u, u=-2\divr(-\Delta)^{-1}S$
\end{proposition}

The strain evolution equation is extremely useful, because it allows us to prove a simplified identity for enstrophy growth, 
which can equivalently be defined in terms of the square of $L^2$ norm of $S, \omega,$ or $\nabla u$ based on an isometry proven by the author in \cites{MillerStrain}.

\begin{proposition} \label{isometry}
For all $-\frac{3}{2}<\alpha<\frac{3}{2},$ 
and for all $S\in\dot{H}^\alpha_{st}$
\begin{equation}
    \|S\|_{\dot{H}^\alpha}^2
    =\frac{1}{2}\|\omega\|_{\dot{H}^\alpha}^2
    =\frac{1}{2}\|\nabla u\|_{\dot{H}^\alpha}^2
\end{equation}
\end{proposition}

\begin{remark}
We should note here that the factor of $\frac{1}{2}$ in Proposition \ref{isometry} is entirely an artifact of how the vorticity is defined. The vorticity is a vector representation of the anti-symmetric part of $\nabla u,$ with
\begin{equation}
A=\frac{1}{2} \left (
\begin{matrix}
0 & \omega_3 & -\omega_2 \\
-\omega_3 & 0 & \omega_1 \\
\omega_2 & -\omega_1 & 0 \\
\end{matrix}
\right ),
\end{equation}
where $A$ is the anti-symmetric part of $\nabla u$ given by
$A_{ij}=\frac{1}{2}\left(\partial_i u_j-\partial_j u_i\right).$
From this identity we can see that
\begin{equation}
    \|S\|_{\dot{H}^\alpha}^2=\|A\|_{\dot{H}^\alpha}^2,
\end{equation}
so the isometry in Proposition \ref{isometry} tells us that all the Hilbert norms of the symmetric and anti-symmetric parts of the gradient of a divergence free vector field are equal.
\end{remark}

\begin{definition} \label{EnergyEnstrophyDef}
Based on the isometry in Proposition \ref{isometry}, we will define the enstrophy of a solution to the Navier--Stokes equation, which can be equivalently expressed by
\begin{align}
    E(t)&=\|S(t)\|_{L^2}^2\\
    &= \frac{1}{2}\|\omega(t)\|_{L^2}^2\\
    &= \frac{1}{2}\|\nabla u(t)\|_{L^2}^2,
\end{align}
and the energy of a solution of the Navier--Stokes equation,
which can be equivalently expressed by
\begin{align}
    K(t)&=\|S(t)\|_{\dot{H}^{-1}}^2\\
    &= \frac{1}{2}\|\omega(t)\|_{\dot{H}^{-1}}^2\\
    &= \frac{1}{2}\|u(t)\|_{L^2}^2,
\end{align}
\end{definition}

\begin{remark}
The energy equality for smooth solutions of the Navier--Stokes equation can be stated in terms of energy and enstrophy as
\begin{equation}
    K(t)+2\int_0^t E(\tau)
    \diff\tau =K_0
\end{equation}
\end{remark}

Enstrophy is a very important quantity because Theorem \ref{KatoMildExistence} states that a smooth solution of the Navier--Stokes equation must exist locally in time for initial data in $u^0\in\dot{H}^1.$ This implies that enstrophy controls regularity, because as long as enstrophy remains bounded on some time interval, a smooth solution can be continued to some later time.

The standard estimate for enstrophy growth is given in terms of nonlocal interaction of the vorticity and the strain:
\begin{equation}
    \frac{\diff}{\diff t} \frac{1}{2}\|\omega(t)\|_{L^2}^2=
    -\|\omega\|_{\dot{H}^1}^2+\left<S,\omega\otimes\omega\right>.
\end{equation}
This is a nonlocal identity because $S$ can be determined in terms of $\omega$ by a nonlocal, zeroth order pseudo-differential operator, with $S=\nabla_{sym}\nabla \times (-\Delta)^{-1}\omega.$ Using the isometry in Proposition \ref{isometry}, and the evolution equations for both the strain and the vorticity, this identity can be drastically simplified, with the nonlocal term replaced by a term involving only the determinant of $S.$

\begin{proposition} \label{NSEnstrophyGrowth}
Suppose $u\in C\left([0,T_{max});\dot{H}^1_{df}\right)$ is a mild solution of the Navier--Stokes equation. Note that this is equivalent to assuming that $S\in C\left([0,T_{max});L^2_{st}\right)$ is a mild solution to the Navier--Stokes strain equation. 
Then for all $0<t<T_{max},$
\begin{align}
    \frac{\diff}{\diff t} \|S(t)\|_{L^2}^2 &=
    -2\|S\|_{\dot{H}^1}^2-\frac{4}{3}\int\tr\left(S^3\right)\\
    &=
    -2\|S\|_{\dot{H}^1}^2-4\int\det(S).
\end{align}
\end{proposition}

This identity was first proven by Neustupa and Penel in \cites{NeustupaPenel1,NeustupaPenel2}.
The analogous result without the dissipation term 
$-2\|S\|_{\dot{H}^1}^2$ was later proven independently by Chae in the context of smooth solutions of the Euler equation in \cite{ChaeStrain} using similar methods to Neustupa and Penel. This identity was also proven using the evolution equation for the strain, a different approach to that of Neustupa and Penel, by the author in \cite{MillerStrain}.
The identity in Proposition \ref{NSEnstrophyGrowth} directly implies a family of scale-invariant regularity criteria in terms of the positive part of the middle eigenvalue of $S.$

\begin{theorem} \label{NSregcrit}
Suppose $u\in C\left([0,T_{max});\dot{H}^1_{df}\right)$ is a mild solution of the Navier--Stokes equation, or equivalently that $S\in C\left([0,T_{max});L^2_{st}\right)$ is a mild solution to the Navier--Stokes strain equation. 
Let $\lambda_1(x,t) \leq \lambda_2(x,t) \leq \lambda_3(x,t)$ be the eigenvalues of $S(x,t),$ and 
let $\lambda^+_2(x,t)=\max\left\{0,\lambda_2(x,t)\right\}.$
Then for all $\frac{3}{q}+\frac{2}{p}=2,
\frac{3}{2}<q\leq +\infty,$ there exists $C_q>0$ depending only on $q$ such that for all $0<t<T_{max}$
\begin{equation}
    \|S(t)\|_{L^2}^2\leq \left\|S^0\right\|_{L^2}^2 \exp\left(C_q\int_0^t
    \left\|\lambda_2^+(\tau)\right\|_{L^q}^p\diff\tau\right).
\end{equation}
In particular, if $T_{max}<+\infty,$ then
\begin{equation}
    \int_0^{T_{max}}\left\|\lambda_2^+(t)\right\|_{L^q}^p\diff t
    =+\infty.
\end{equation}
\end{theorem}

This regularity criterion was first proven by Neustupa and Penel in \cites{NeustupaPenel1,NeustupaPenel2,NeustupaPenel3}. It was also proven independently by the author in \cite{MillerStrain}.
Note that because $\tr(S)=0,$ this regularity criterion significantly restricts the geometry of any finite-time blowup for the Navier--Stokes equation: any blowup must be driven by unbounded planar stretching and axial compression, with the strain having two positive eigenvalues, and one very negative eigenvalue.

There are many other conditional regularity results, which guarantee the regularity of solutions as long as some scale critical quantity remains finite, including regularity criteria involving the derivative in just one direction $\partial_3 u$ \cite{Kukavica}, and involving just one velocity direction $u_3$ \cites{CheminOneDirection,CheminOneDirection2}.
For a more thorough, but by no means exhaustive, treatment of regularity criteria for the Navier--Stokes equation see Chapter 11 in \cite{NS21}. 

In this paper, we will take the opposite approach. We will prove finite-time blowup for solutions of the Navier--Stokes equation with a fairly broad set of initial data, assuming that a certain scale invariant quantity related to the structure of the nonlinearity remains small. We will do this first by considering a model equation for the Navier--Stokes strain equation and proving finite-time blowup for solutions of this this model equation, and then by viewing the actual Navier--Stokes strain equation as a perturbation of the model equation.

In order to do this, we will drop the advection and the vorticity terms from the evolution equation \eqref{NSstrain} entirely, along with a piece of the $S^2$ term so that the enstrophy growth identity in Proposition \ref{NSEnstrophyGrowth} still holds.
We will show that
\begin{equation} \label{StrainVort}
    \left<S,\omega\otimes\omega\right>
    =-4\int\det(S)
    =-\frac{4}{3}\left<S^2,S\right>,
\end{equation}
and
\begin{equation}
    \left<(u\cdot\nabla)S,S\right>=0,
\end{equation}
and therefore
\begin{equation}
    \left<P_{st}\left((u\cdot\nabla)S+\frac{1}{3}S^2
    +\frac{1}{4}\omega\otimes\omega\right), S\right>=0.
\end{equation}
Using this identity, we can rewrite the full Navier--Stokes strain equation as
\begin{equation}
    \partial_t S-\Delta S+\frac{2}{3}P_{st}\left(S^2\right)
    +P_{st}\left((u\cdot\nabla)S+\frac{1}{3}S^2
    +\frac{1}{4}\omega\otimes\omega \right)=0,
\end{equation}
Dropping the term $P_{st}\left((u\cdot\nabla)S+\frac{1}{3}S^2
+\frac{1}{4}\omega\otimes\omega\right)$ from the evolution equation,
our strain model equation will be given by
\begin{equation} \label{StrainModel}
    \partial_t S -\Delta S 
    +\frac{2}{3}P_{st}\left(S^2\right)=0.
\end{equation}
We will refer to \eqref{StrainModel} as the strain self-amplification model equation, because it isolates the interaction of the strain with itself, discarding the nonlocal interaction with the vorticity and the effects of advection.
In the model equation, we are dropping a combination of terms that are orthogonal to $S$ in $L^2,$ while keeping the two terms that contribute to the evolution in time of the $L^2$ norm to first order.
We will also show that for solutions of the strain self-amplification model equation
\begin{equation}
    \frac{\diff}{\diff t} \|S(t)\|_{L^2}^2=
    -2\|S\|_{\dot{H}^1}^2
    -4\int\det(S),
\end{equation}
so the strain self-amplification model equation does in fact have the same identity for enstrophy growth as the Navier--Stokes equation, and consequently has a regularity criterion for $\lambda_2^+$ in the critical Lebesgue spaces $L^p_t L^q_x$ entirely analogous to the regularity criterion for the Navier--Stokes equation in Theorem \ref{NSregcrit}.

Solutions of this model equation blowup in finite-time for a fairly wide range of initial conditions.
\begin{theorem} \label{ModelBlowupIntro}
    Suppose $S\in C\left([0,T_{max});H^1_{st}\right)$ is a mild solution of the strain self-amplification model equation, such that
\begin{equation}
    -3\left\|S^0\right\|_{\dot{H}^1}^2
    -4\int\det\left(S^0\right)>0.
\end{equation}
Then for all $0<t<T_{max}$,
\begin{equation} \label{IntroLowerBound}
    E(t)>\frac{E_0}{(1-r_0 t)^2},
\end{equation}
where
\begin{equation}
    r_0=\frac{-3\left\|S^0\right\|_{\dot{H}^1}^2
    -4\int\det\left(S^0\right)}
    {2\left\|S^0\right\|_{L^2}^2}.
\end{equation}
Note in particular that this implies 
\begin{equation}
    T_{max} \leq \frac{2\left\|S^0\right\|_{L^2}^2}
    {-3\left\|S^0\right\|_{L^2}^2
    -4\int\det\left(S^0\right)}.
\end{equation}
Furthermore, for all $\frac{2}{p}+\frac{3}{q}=2,
\frac{3}{2}<q\leq +\infty,$
\begin{equation}
    \int_0^{T_{max}}\|\lambda_2^+(t)\|_{L^q}^p
    \diff t=+\infty.
\end{equation}
\end{theorem}

\begin{remark}
The key to the proof of Theorem \ref{ModelBlowupIntro}, which is the main result of this paper, is a $\frac{3}{2}$ lower bound on the rate of enstrophy growth for a wide range of initial conditions.
In particular, we will show that 
if
\begin{equation}
    g_0:=\frac{-3\left\|S^0\right\|_{\dot{H}^1}^2
    -4\int\det\left(S^0\right)}
    {\left\|S^0\right\|_{L^2}^3}>0,
\end{equation}
then for all $0 < t<T_{max}$,
\begin{equation}
    \frac{\diff}{\diff t} E(t)
    >
    g_0 E(t)^\frac{3}{2},
\end{equation}
which immediately yields the estimate \eqref{IntroLowerBound} in Theorem \ref{ModelBlowupIntro}.
\end{remark}

\begin{remark}
Theorem \ref{ModelBlowupIntro} shows that the regularity criterion in Theorem \ref{NSregcrit}, which guarantees the existence of smooth solutions of the Navier--Stokes equation so long as $\lambda_2^+\in L_t^p L_x^q,$ is not enough to guarantee the global existence of smooth solutions to the Navier--Stokes equation just by making use of the constraint space. 
For solutions of the strain self-amplification model equation, which is an evolution equation on the constraint space $L^2_{st},$ $\lambda_2^+$ becomes unbounded in this whole family of scale critical spaces. The regularity criterion on $\lambda_2^+$ implies that blowup for the Navier--Stokes equation must be characterized by unbounded planar stretching and axial compression, corresponding to a strain matrix with two positive eigenvalues and one very negative eigenvalue in turbulent regions. One physical example of such a structure in turbulent fluids is two colliding jets. The blowup result for the strain self-amplification model equation shows that blowup with these features is possible within the relevant constraint space.
\end{remark}

Because we chose our strain self-amplification model equation \eqref{StrainModel} by dropping some terms from the full strain equation, we can prove a new conditional blowup result for the full Navier--Stokes equation, by viewing the actual strain equation as a perturbation of the strain self-amplification model equation.

\begin{theorem} \label{PerturbativeBlowupIntro}
Suppose $u\in C\left([0,T_{max});H^2_{df}\right)$ is a mild solution of the Navier--Stokes equation such that
\begin{equation}
    f_0:=-3\left\|S^0\right\|_{\dot{H}^1}^2
    -4\int\det\left(S^0\right)>0,
\end{equation}
and for all $0<t<T_{max}$
\begin{equation}
    \frac{\left\|P_{st}\left((u\cdot\nabla)S+\frac{1}{3}S^2
    +\frac{1}{4}\omega\otimes\omega
    \right)(\cdot,t)\right\|_{L^2}}
    {\left\|\left(-\Delta S+P_{st}\left(
    \frac{1}{2}(u\cdot\nabla)S+\frac{5}{6}S^2
    +\frac{1}{8}\omega\otimes\omega\right)
    \right)(\cdot,t)\right\|_{L^2}}
    \leq
    2.
\end{equation}
Then there is finite-time blowup with
\begin{equation}
    T_{max}
    <
    T_*
    :=
    \frac{-E_0+\sqrt{E_0^2+f_0 K_0}}
    {f_0},
\end{equation}
where $K_0$ and $E_0$ are taken as in Definition \ref{EnergyEnstrophyDef},
and $f_0$ is as defined above.
\end{theorem}

\begin{remark}
Theorem \ref{PerturbativeBlowupIntro} quantifies how close solutions of the Navier--Stokes strain equation have to be to solving the model equation in order to be guaranteed to blowup in finite-time. This result is---to the knowledge of the author---the first of its kind. There are many results stating that if some scale invariant quantity is finite, then solutions of the Navier--Stokes equation must be smooth, such as the aforementioned Ladyzhenskaya-Prodi-Serrin and Beale-Kato-Majda regularity criteria. Theorem \ref{PerturbativeBlowupIntro} is the first result to say that, for some set of initial data, if a scale invariant quantity remains small enough for the history of the solution, there must be blowup in finite-time.
\end{remark}

\begin{remark}
The mechanism for blowup proposed in Theorem \ref{PerturbativeBlowupIntro} for the Navier--Stokes equation is also consistent with research on the turbulent energy cascade. Very recently, Carbone and Bragg showed both theoretically and numerically that strain self-amplification is a more important factor in the average turbulent energy cascade than vortex stretching \cite{CarboneBragg}. This gave a concrete statement to a line of inquiry on the turbulent energy cascade begun by Tsinober in \cite{Tsinober}. The turbulent energy cascade is directly tied to the Navier--Stokes regularity problem, as finite-time blowup requires a transfer of energy to arbitrarily small scales, so this suggests that the self-amplification of strain is the most likely potential mechanism for the finite-time blowup of solutions of the Navier--Stokes equation. The conditional blowup result in this paper gives a quantitative estimate on the structure of the nonlinearity that will lead to finite-time blowup for the Navier--Stokes equation via the self-amplification of strain if it is maintained by the dynamics.
\end{remark}

\begin{remark}
Turbulent solutions of the Navier--Stokes equation are, almost by definition, difficult to impossible to write down in closed form. This poses a significant barrier to proving the existence of smooth solutions of the Navier--Stokes equation that blowup in finite-time: if finite-time blowup solutions do in fact exist, it will still almost certainly not be possible to give a negative answer to the Navier--Stokes regularity problem by providing a counterexample in closed form.
Any progress on the Navier--Stokes regularity problem in the direction of proving the existence of finite-time blowup will likely require an interplay of analysis and numerics.
Theorem \ref{PerturbativeBlowupIntro} provides a quantitative criterion that could guide further numerical work searching for possible blowup solutions.
\end{remark}

We cannot show that there are any solutions of Navier--Stokes equation which satisfy the perturbative condition in Theorem \ref{PerturbativeBlowupIntro} up until $T^*$. If we could, then this would solve the Navier--Stokes regularity problem by implying the existence of finite-time blowup. We can, however, use scaling arguments to prove that this condition is satisfied for short times for some solutions of the Navier--Stokes equation.

\begin{theorem} \label{PerturbativeShortTimeIntro}
There exists a mild solution of the 
Navier--Stokes equation 
$u\in C\left([0,T_{max});H^3_{df}\right)$
and $\epsilon>0$ such that
\begin{equation}
    -3\left\|S^0\right\|_{\dot{H}^1}^2
    -4\int\det\left(S^0\right)>0
\end{equation}
and for all $0\leq t<\epsilon$
\begin{equation}
    \frac{\left\|P_{st}\left((u\cdot\nabla)S+\frac{1}{3}S^2
    +\frac{1}{4}\omega\otimes\omega
    \right)(\cdot,t)\right\|_{L^2}}
    {\left\|\left(-\Delta S+P_{st}\left(
    \frac{1}{2}(u\cdot\nabla)S+\frac{5}{6}S^2
    +\frac{1}{8}\omega\otimes\omega\right)
    \right)(\cdot,t)\right\|_{L^2}}
    \leq
    2.
\end{equation}
\end{theorem}

\begin{remark}
In this paper we have taken the viscosity to be $\nu=1$. For the Navier--Stokes regularity problem, we can fix the viscosity to be $\nu=1$ without loss of generality, because it is equivalent up to rescaling to the Navier--Stokes regularity problem for arbitrary $\nu>0.$ It is useful, however, to see how the blowup results that we will prove scale with respect to the viscosity parameter $\nu>0.$ If we take the viscosity to be $\nu>0,$ then the Navier--Stokes equation is now given by
\begin{equation}
    \partial_t u-\nu\Delta u+P_{df}\nabla \cdot 
    (u\otimes u)=0,
\end{equation}
and the strain self-amplification model equation is given by
\begin{equation}
    \partial_t S-\nu\Delta S
    +\frac{2}{3}P_{st}\left(S^2\right)=0.
\end{equation}

In this case, the condition
\begin{equation}
    -3\left\|S^0\right\|_{\dot{H}^1}^2
    -4\int\det\left(S^0\right)>0,
\end{equation}
in Theorems \ref{ModelBlowupIntro} and \ref{PerturbativeBlowupIntro} 
is replaced with the condition
\begin{equation}
    -3\nu\left\|S^0\right\|_{\dot{H}^1}^2
    -4\int\det\left(S^0\right)>0.
\end{equation}
Likewise the condition
\begin{equation}
    \frac{\left\|P_{st}\left((u\cdot\nabla)S+\frac{1}{3}S^2
    +\frac{1}{4}\omega\otimes\omega
    \right)(\cdot,t)\right\|_{L^2}}
    {\left\|\left(-\Delta S+P_{st}\left(
    \frac{1}{2}(u\cdot\nabla)S+\frac{5}{6}S^2
    +\frac{1}{8}\omega\otimes\omega\right)
    \right)(\cdot,t)\right\|_{L^2}}
    \leq
    2.
\end{equation}
in Theorems \ref{PerturbativeBlowupIntro} and \ref{PerturbativeShortTimeIntro}
is replaced by
\begin{equation}
    \frac{\left\|P_{st}\left((u\cdot\nabla)S+\frac{1}{3}S^2
    +\frac{1}{4}\omega\otimes\omega
    \right)(\cdot,t)\right\|_{L^2}}
    {\left\|\left(-\nu \Delta S+P_{st}\left(
    \frac{1}{2}(u\cdot\nabla)S+\frac{5}{6}S^2
    +\frac{1}{8}\omega\otimes\omega\right)
    \right)(\cdot,t)\right\|_{L^2}}
    \leq
    2.
\end{equation}
\end{remark}

\begin{remark}
We should note in particular this means that if $S^0\in H^1_{st}$ and $-\int\det\left(S^0\right)>0,$ then for all
\begin{equation}
    0<\nu<\nu^0:=
    \frac{-4\int\det\left(S^0\right)}
    {3\left\|S^0\right\|_{\dot{H}^1}^2},
\end{equation}
the strain self-amplification model equation with viscosity $\nu$ blows up in finite-time.
This implies that blowup for the strain self-amplification model equation is generic
at sufficiently large Reynolds number,
subject only to the geometric sign constraint on initial data,
that $-\int\det\left(S^0\right)>0$.
This suggests that the self-amplification of strain is likely the driving factor behind possible blowup for the full Navier--Stokes equation, and any depletion of nonlinearity preventing finite-time blowup must come from the effects of advection and the nonlocal interaction of strain and vorticity.

This also means that finite-time blowup may occur for the strain self-amplification model equation even in simplified geometric settings where blowup is ruled out for the full Navier--Stokes equation. We will show that there is finite-time blowup for the strain self-amplification model equation even when restricted to axisymmetric, swirl-free solutions. This contrasts strongly with the Navier--Stokes equation where there is global regularity for arbitrarily large initial data in the axisymmetric, swirl-free case. There are also axisymmetric, swirl-free solutions of the Navier--Stokes equation that satisfy the perturbative condition for short times, as in Theorem \ref{PerturbativeShortTimeIntro}. 
Such solutions cannot, of course, satisfy the conditions of Theorem \ref{PerturbativeBlowupIntro}, because they cannot blowup in finite-time, and hence the perturbative condition can only be satisfied for short times in such cases.
\end{remark}

\begin{remark}
Because such a wide range of initial data lead to finite-time blowup for the strain self-amplification model equation, the set of initial data for which there is finite-time blowup for this model equation is too broad a set to consider as possible candidates for finite-time blowup for the Navier--Stokes equation.
While initial data that blowup in finite-time are ubiquitous at high Reynolds number, subject only to a sign constraint on the integral of the determinant of the strain, this does not necessarily mean that blowup itself is generic. There could be certain structures or scaling laws that emerge as the blowup time is approached for any blowup solution; further study is needed.

One possible avenue for further work would be to allow the dynamics of the strain self-amplification model equation to select candidates for blowup for the full Navier--Stokes equation. 
Consider a solution of the strain self-amplification model equation that is not axisymmetric, swirl-free, and that blows up in finite-time $T_{max}<+\infty$. 
If we take $S(\cdot,T_{max}-\epsilon)$ 
for some $0<\epsilon \ll T_{max}$ 
as our initial data for the full Navier--Stokes equation, then this would be a very natural candidate for blowup if blowup does in fact occur for the full Navier--Stokes equation.
To consider such an approach, more detailed study of the qualitative features of blowup solutions of the strain self-amplification model equation is needed. At present, essentially all we know about such solutions is a lower bound on the growth of enstrophy and that $\lambda_2^+$ blows up in the scale-critical $L_t^p L_x^q$ spaces.
\end{remark}

In section 2, we will discuss the relationship between our results and previous results for simplified model equations for Navier--Stokes. In section 3, we will define a number of the spaces used in our analysis and give precise definitions of mild solutions. In section 4, we will develop the local well-posedness theory for the strain self-amplification model equation, including proving global well-posedness for small initial data, and scale critical regularity criteria in terms of $\lambda_2^+$ and in terms of two vorticity components.
In section 5, we will prove Theorem \ref{ModelBlowupIntro}, demonstrating the existence of finite-time blowup for solutions of the strain self-amplification model equation, and will prove a number of properties about the set of initial data satisfying the hypothesis of this theorem.
Finally in section 6, we will prove Theorem \ref{PerturbativeBlowupIntro}, the conditional blowup result for the full Navier--Stokes equation when a perturbative condition is satisfied by the history of the solution, and further show that this perturbative condition is satisfied for short times for some solutions of the Navier--Stokes equation.

\section{Relationship to previous literature}

There are a number of previous results that prove blowup for simplified model equations for Navier--Stokes with the hope of elucidating possibilities of extending this to the full Navier--Stokes equation. Montgomery-Smith introduced a scalar toy model equation, replacing the first order pseudo-differential operator $P_{df}\nabla \cdot$ by $-(-\Delta)^{\frac{1}{2}},$ and replacing the quadratic term $u \otimes u$ by $u^2,$ giving the scalar equation
\begin{equation}
    \partial_t u-\Delta u -(-\Delta)^{\frac{1}{2}}
    \left(u^2\right)=0,
\end{equation}
and proved the existence of finite-time blowup solutions for this equation \cite{MontgomerySmith}.
This blowup result was extended by Gallagher and Paicu to a model equation on the space of divergence free vector fields by adjusting the Fourier symbol of the first order pseudo differential operator \cite{GallagherToy}. However, while Gallagher and Paicu's model equation is an evolution equation on natural constraint space, the space of divergence free vector fields, neither of these model equations respects the energy equality, and so both are still quite far from the actual fluid equations. They are nonetheless important in that they establish that it is not possible to prove global regularity for the Navier--Stokes equation using heat semi-group methods alone.

Tao improved on these earlier blowup results by introducing a Fourier space averaged Navier--Stokes model equation \cite{TaoModel}. His model equation is given by
\begin{equation} \label{TaoEq}
    \partial_t u-\Delta u+\Tilde{B}(u,u)=0,
\end{equation}
where $\Tilde{B}(u,u)$ is a Fourier space averaged version of 
$P_{df}\nabla \cdot(u\otimes u)$.
This equation is an improvement over the previous results because $\Tilde{B}$ is constructed so that
\begin{equation}
    \left<\Tilde{B}(u,u),u\right>=0,
\end{equation}
so Tao's model equation \eqref{TaoEq} respects the energy equality, with for all $0<t<T_{max}$,
\begin{equation}
  \frac{1}{2}\|u(t)\|_{L^2}^2+
  \int_0^t\|u(\tau)\|_{\dot{H}^1}^2 \diff\tau
  =\frac{1}{2}\left\|u^0\right\|_{L^2}^2,
\end{equation}
while also exhibiting finite-time blowup. The operator $\Tilde{B}$ also has some of the same harmonic analysis bounds as those found for full Navier--Stokes equation, in particular
\begin{equation} \label{TaoBound}
    \|\Tilde{B}(u,u)\|_{L^2}\leq C \|u\|_{L^4}\|\nabla u\|_{L^4}.
\end{equation}
The fact that there are finite-time blowup solutions to Tao's model equation shows that if there is global regularity for solutions of the Navier--Stokes equation with arbitrary smooth initial data, the proof will require more than the energy equality and the standard harmonic analysis techniques.
New a priori bounds are needed.
We will note in particular that the bound in \eqref{TaoBound} implies that Tao's model equation respects the Ladyzhenskaya-Prodi-Serrin regularity criterion, that is if $T_{max}<+\infty$ for a solution $u$ of \eqref{TaoEq}, then for all $\frac{2}{p}+\frac{3}{q}=1, 3<q\leq +\infty,$
\begin{equation}
    \int_0^{T_{max}}\|u\|_{L^q}^p=+\infty.
\end{equation}

While the Tao model equation respects the energy equality and some of the structure of the velocity equation, it does not respect the structure of the vorticity or strain equations. In particular, Tao's model does not respect---or at least has not been shown to respect---the identity for enstrophy growth in Proposition \ref{NSEnstrophyGrowth}, the regularity criterion on $\lambda_2^+$ in Theorem \ref{NSregcrit}, or the regularity criterion on two components of the vorticity in \eqref{2vort}.
The finite-time blowup result for the strain self-amplification model equation is an advance on Tao's model equation if the Navier--Stokes regularity problem is considered from the point of view of enstrophy growth. The model equation considered here, unlike Tao's model equation, does not respect the energy equality; however, from a mathematical point of view, the energy equality is less fundamental to the Navier--Stokes regularity problem than the identity for enstrophy growth, because energy does not control regularity. 
Blowup for the Navier--Stokes equation in finite-time is equivalent to the blowup of enstrophy in finite-time, so mathematically it is very significant that we are able to show blowup for an evolution equation on $L^2_{st}$ that respects the identity for enstrophy growth in Proposition \ref{NSEnstrophyGrowth}.
In summary, Tao's model equation reflects more of the structure of the velocity formulation of the Navier--Stokes regularity problem, while the strain self-amplification model equation reflects more of the structure of both the strain and vorticity formulations of the regularity problem.

The strain self-amplification model equation is the first model equation of possible Navier--Stokes blowup that respects regularity criteria for the Navier--Stokes equation based not just on size, but on geometric structure as well. It is straightforward to show that the strain self-amplification, Montgomery-Smith, Gallagher-Paicu, and Tao model equations all respect the Ladyzhenskaya-Prodi-Serrin regularity criterion on the size of $u$, but the strain self-amplification model equation also respects the regularity criterion on $\lambda_2^+$ in Theorem \ref{NSregcrit}. 
This implies as a corollary that the strain self-amplification model equation must respect the regularity criterion on two vorticity components proven by Chae and Choe as well.
This suggests it captures significantly more of the geometry of potential Navier--Stokes blowup than any of the previous model equations, at least as far as deformation and vorticity are concerned.

Theorem \ref{PerturbativeBlowupIntro} shows that the local part of the nonlinearity of the strain evolution equation tends to lead to finite-time blowup for a wide range of initial conditions, so there must be finite-time blowup for the Navier--Stokes equation similar to the blowup for the model equation for the self-amplification of strain unless the vorticity and advection terms act to deplete this nonlinearity and prevent blowup.
This is consistent with a number of previous works for model equations related to the Navier--Stokes and Euler equations that suggest that advection plays a regularizing role.
For instance, there are a number of previous works on the Constantin-Lax-Majda \cite{ConstantinLaxMajda} and De Gregorio \cite{DeGregorio} 1D models for the vorticity equation, which showed that advection may have a regularizing effect \cites{DiegoCordoba,ElgindiAdvection,Sverak1DModels}.
Theorem \ref{ModelBlowupIntro}, which states that finite-time blowup occurs for a wide range of initial data for the strain self-amplification model equation, extends the analysis of the regularizing role of advection from 1D models that do not respect the structure of the constraint space, to a 3D model that does respect the structure of the constraint space.

There is also previous research on model equations for the axisymmetric Navier--Stokes and Euler equations which preserve more of the structures of three dimensional fluid mechanics than the Constantin-Lax-Majda or De Gregorio models. These model equations also show that advection plays a regularizing role \cites{HouZhen,HouLiuWang}. Furthermore, there has been research on the possible role of advection in the depletion of nonlinearity related to its interaction with the pressure in the growth of subcritical $L^q$ norms of $u$ \cite{TranYu}. Theorem \ref{PerturbativeBlowupIntro} is entirely novel, however, because it is the first perturbative, finite-time blowup result
related to the possible role of nonlinear depletion by advection.
The previous results were either heuristic or numerical;
in contrast, Theorem \ref{PerturbativeBlowupIntro} provides a quantitative condition guaranteeing blowup as long as the terms which could potentially deplete the nonlinear self-amplification of strain are small enough relative to strain self-amplification.

Finally, we should mention that very recently, Elgindi \cite{ElgindiFiniteTimeEuler} and Elgindi, Ghoul, and Masmoudi \cite{ElgindiGhoulMasmoudi} proved finite-time blowup for a class of $C^{1,\alpha}\left(\mathbb{R}^3\right)$ solutions of the Euler equation that conserve energy. While the question of blowup for smooth solutions to the Euler equation remains open, this represents an enormous step forward in providing an example of classical solutions to the Euler equation that blowup in finite-time. The blowup solutions of the Euler equation constructed in \cites{ElgindiFiniteTimeEuler,ElgindiGhoulMasmoudi} are axisymetric and swirl-free, and are closely related to an example of finite-time blowup that we will construct for the strain self-amplification model equation. We will discuss this further in section 5.

\section{Definitions}

We begin by defining the homogeneous and inhomogeneous Hilbert spaces.
\begin{definition}
For all $s \in \mathbb{R},
H^s\left(\mathbb{R}^3\right)$ 
will be the Hilbert space with norm
\begin{equation}
    \|f\|_{H^s}^2=\int_{\mathbb{R}^3} \left(1+
    4 \pi^2|\xi|^2 \right)^s
    |\hat{f}(\xi)|^2 d\xi=
    \left\|\left(1+4 \pi^2|\xi|^2 \right)^\frac{s}{2}
    \hat{f}\right\|_{L^2}^2, 
\end{equation}
and for all $-\frac{3}{2}<s<\frac{3}{2},$
$\dot{H}^s\left(\mathbb{R}^3\right)$
will be the homogeneous Hilbert space with norm
\begin{equation}
    \|f\|_{\dot{H}^s}^2=\int_{\mathbb{R}^3} 
    (2 \pi|\xi|)^{2s} |\hat{f}(\xi)|^2 d\xi=
    \left\|(2 \pi|\xi|)^{s}
    \hat{f}\right\|_{L^2}^2.
\end{equation}
\end{definition}

Note that when referring to
$H^s\left(\mathbb{R}^3\right), \dot{H}^s\left(\mathbb{R}^3\right), \text{or} \:
L^p\left(\mathbb{R}^3 \right),$ the $\mathbb{R}^3$ will often be omitted for brevity's sake. All Hilbert and Lebesgue norms are taken over $\mathbb{R}^3$ unless otherwise specified.
Furthermore, $\mathbb{S}^{3\times3}$ will refer to the space of three by three symmetric matrices,
\begin{equation}
    \mathbb{S}^{3 \times 3}
    =
    \left\{ 
    \left(\begin{array}{ccc}
         a & d & e  \\
         d & b& f  \\
         e & f& c
    \end{array}\right):
    a,b,c,d,e,f\in\mathbb{R}
    \right\}.
\end{equation}
We now define the subspaces of divergence free vector fields and strain matrices in Hilbert spaces.

\begin{definition} \label{DivFree}
    For all $s\in \mathbb{R}$ define $H^s_{df}\subset
    H^s\left(\mathbb{R}^3;\mathbb{R}^3 \right)$ by
    \begin{equation}
H^s_{df}=\left\{u\in H^s\left(\mathbb{R}^3;\mathbb{R}^3
\right): \xi \cdot \hat{u}(\xi)=0, \text{almost everywhere } \xi\in\mathbb{R}^3 \right\}.
    \end{equation}
    For all $-\frac{3}{2}<s<\frac{3}{2},$ define $\dot{H}^s_{df}\subset
    \dot{H}^s\left(\mathbb{R}^3;\mathbb{R}^3 \right)$ by
    \begin{equation}
\dot{H}^s_{df}=\left\{u\in \dot{H}^s\left(\mathbb{R}^3;\mathbb{R}^3
\right): \xi \cdot \hat{u}(\xi)=0, \text{almost everywhere } \xi\in\mathbb{R}^3 \right\}.
\end{equation}
\end{definition}

\begin{definition} \label{HilbertStrain}
    For all $s\in \mathbb{R}$ define $H^s_{st}\subset
    H^s\left(\mathbb{R}^3;\mathbb{S}^{3\times 3}
    \right)$ by
    \begin{equation}
H^s_{st}=
\left\{\nabla_{sym}(-\Delta)^{-\frac{1}{2}}u
: u\in H^s_{df}\right\}.
    \end{equation}
    For all $-\frac{3}{2}<s<\frac{3}{2},$ define $\dot{H}^s_{st}\subset
    \dot{H}^s\left(\mathbb{R}^3;\mathbb{S}^{3\times 3} \right)$ by
    \begin{equation}
\dot{H}^s_{st}=
\left\{\nabla_{sym}(-\Delta)^{-\frac{1}{2}}u
: u\in \dot{H}^s_{df}\right\}.
\end{equation}
\end{definition}

\begin{definition} \label{LqStrain}
For all $1<q<+\infty$, define $L^q_{st}$ by
\begin{equation}
    L^q_{st}=\left\{ 
    S\in L^q\left(\mathbb{R}^3;\mathbb{S}^{3\times 3}\right):
    \tr(S)=0, S=\nabla_{sym}(-\Delta)^{-1}(-2\divr(S))
    \right\}.
\end{equation}
\end{definition}

\begin{remark}
We will note that we have already defined $L^2_{st}$ in the introduction, so we now have two definitions of $L^2_{st}$.
These definitions are equivalent, as was proven by the author in \cite{MillerStrain}. 
The key reason for this is that, just as the vorticity can be inverted to obtain the velocity, with
\begin{equation}
    u= \nabla\times(-\Delta)^{-1}\omega,
\end{equation}
so too can the strain be inverted to obtain the velocity, with
\begin{equation}
    u=-2\divr (-\Delta)^{-1}S.
\end{equation}
This means that for all $S\in L^2_{st},$
\begin{equation}
    S=\nabla_{sym}u
\end{equation}
if and only if
\begin{equation}
    u=-2\divr (-\Delta)^{-1}S.
\end{equation}
This implies the condition in Definition \ref{LqStrain} in the case $q=2$ is equivalent to the the condition in Definition \ref{L2Strain}. See \cite{MillerStrain} for more details.
We will also note that Definition \ref{LqStrain} is well defined because the operator $-2\nabla_{sym}\divr(-\Delta)^{-1}$ is a bounded linear operator mapping 
$L^q \longrightarrow L^q$ for all $1<q<+\infty$.
This follows from the boundedness of the Riesz transform
$R=\nabla(-\Delta)^{-\frac{1}{2}}$ because
\begin{equation}
    -2\nabla_{sym}\divr(-\Delta)^{-1}S
    =-2R_{sym} R\cdot S.
\end{equation}
\end{remark}

We will also define axisymmetric, swirl-free vector fields and strain matrices.
\begin{definition}
Begin by letting
\begin{align}
    r&=\sqrt{x_1^2+x_2^2} \\
    z&=x_3 \\
    e_r&= \frac{1}{r}(x_1,x_2,0) \\
    e_z&= (0,0,1).
\end{align}
We will say that $u\in \dot{H}^1_{df}$ is an axisymmetric, swirl-free vector field if
\begin{equation}
    u(x)=u_r(r,z)e_r+u_z(r,z)e_z.
\end{equation}
Note that the divergence free condition can be expressed in this case by
\begin{equation}
    \nabla\cdot u=\partial_r u_r+\frac{1}{r}u_r+
    \partial_z u_z=0.
\end{equation}
We will say that $S\in L^2_{st}$ is an axisymmetric, swirl-free strain matrix, if
\begin{equation}
    S=\nabla_{sym}u,
\end{equation}
where $u\in\dot{H}^1_{df}$ is an axisymmetric, swirl-free vector field.
\end{definition}

We conclude this section by providing the precise definitions for mild solutions of the Navier--Stokes equation, the Navier--Stokes strain equation, and the strain self-amplification model equation.

\begin{definition}
A velocity field $u\in C\left([0,T_{max});\dot{H}^1_{df}\right)$ is a mild solution to the Navier--Stokes equation if it satisfies \eqref{NSprojection} in the sense of Duhamel's formula, that is, if for all
$0<t<T_{max}$
\begin{equation}
    u(t)=e^{t\Delta}u^0-\int_{0}^t e^{\tau\Delta}
    P_{df}\nabla \cdot(u\otimes u)(t-\tau)\diff\tau.
\end{equation}
Note that $e^{t\Delta}$ is defined in terms of convolution with the heat kernel
\begin{equation}
    G(x,t)=\frac{1}{(4\pi t)^\frac{3}{2}}
    \exp\left(-\frac{|x|^2}{4t}\right)
\end{equation}
so that
\begin{equation}
    e^{t\Delta}f=G(\cdot,t)*f.
\end{equation}
\end{definition}

\begin{remark}
$T_{max}$ is the maximal time of existence for a mild solution. If there is a mild solution globally in time for some initial data $u^0\in \dot{H}^1_{df}$, then $T_{max}=+\infty,$ and if there is not a mild solution globally in time, then $T_{max}<+\infty$ is the blowup time when the solution becomes singular.
\end{remark}

\begin{definition}
A strain matrix $S\in C\left([0,T_{max});L^2_{st}\right)$ is a mild solution to the Navier--Stokes strain equation if it satisfies \eqref{NSstrain} in the sense of Duhamel's formula, that is, if for all
$0<t<T_{max}$
\begin{equation}
    S(t)=e^{t\Delta}S^0-\int_{0}^t e^{\tau\Delta}
    P_{st}\left((u\cdot\nabla)S+S^2+\frac{1}{4}\omega\otimes\omega
    \right)(t-\tau)\diff\tau,
\end{equation}
with $u=(-\Delta)^{-1}\left(-2 \divr(S)\right)$ and
$\omega=\nabla \times u.$
\end{definition}

\begin{definition}
$S\in C\left([0,T_{max});L^2_{st}\right)$ is a mild solution to the strain self-amplification model equation \eqref{StrainModel} if $S$ satisfies this equation in the sense of Duhamel's formula, that is for all $0<t<T_{max}$
\begin{equation}
    S(t)=e^{t\Delta}S^0-\frac{2}{3}\int_{0}^t e^{\tau\Delta}
    P_{st}\left(S^2\right)(t-\tau)\diff\tau
\end{equation}
\end{definition}

\section{Some properties of strain self-amplification model equation}

We begin this section by considering the local-in-time existence of mild solutions to the strain self-amplification model equation.

\begin{theorem} \label{ModelMildExistence}
Let $C=\left(\frac{3}{32\|g\|_{L^2}}\right)^4,$ where
    $g(x)=\frac{1}{(2\pi)^\frac{3}{2}}
    \exp\left(-\frac{|x|^2}{4}\right).$
For all $S^0\in L^2_{st}$, 
there exists a unique mild solution to the strain self-amplification model equation, 
$S\in C\left([0,T_{max});
L^2_{st}\right),$ 
with $T_{max}\geq \frac{C}{\left\|S^0\right\|_{L^2}^4}$.
Furthermore, $S\in C\left((0,T_{max});H^\infty\right),$ 
and is therefore smooth for all positive times up until possible blowup.
\end{theorem}

\begin{proof}
The proof of Theorem \ref{ModelMildExistence} is essentially the same as the proof of local existence of mild solutions for the Navier--Stokes equation introduced by Kato and Fujita. It will be based on a Banach fixed point argument.

We begin by fixing
\begin{equation}
    T<\frac{C}{\left\|S^0\right\|_{L^2}^4}.
\end{equation}
Note that this implies that
\begin{equation}
    \frac{32}{3}\|g\|_{L^2}\left\|S^0\right\|_{L^2}  T^\frac{1}{4}<1.
\end{equation}
Define the map 
$W:C\left([0,T];L^2_{st}\right) \to
C\left([0,T];L^2_{st}\right),$ by
\begin{equation}
    W[M](t)=e^{t\Delta}S^0+\int_{0}^t e^{\tau \Delta}
    P_{st}\left(M^2(t-\tau)\right)\diff\tau.
\end{equation}
Note that $S$ being a mild solution of the heat equation is equivalent to $S$ being a fixed point of this map with
$W[S]=S.$

We will first show that if
$\|M\|_{C\left([0,T];L^2\right)}\leq 2\left\|S^0\right\|_{L^2},$ then
$\|W[M]\|_{C\left([0,T];L^2\right)}\leq 2\left\|S^0\right\|_{L^2}.$
Recall that
\begin{equation}
    e^{t\Delta}f=G(\cdot,t)*f,
\end{equation}
where 
\begin{equation}
    G(x,t)=t^{-\frac{3}{2}}g(t^{-\frac{1}{2}}x).
\end{equation}
Therefore we can compute that
\begin{equation}
    \|G(\cdot,t)\|_{L^2}=\|g\|_{L^2}t^{-\frac{3}{4}}.
\end{equation}
Applying Young's inequality for convolutions we find that for all $0\leq t\leq T$
\begin{align}
    \|W[M](t)\|_{L^2}&\leq \left\|S^0\right\|_{L^2}+ \frac{2}{3}\int_0^t
    \left\|P_{st}\left(e^{\tau\Delta})M^2\right)
    \right\|_{L^2} \diff\tau\\
    &\leq
    \left\|S^0\right\|_{L^2}+ \frac{2}{3}\int_0^t \|G(\cdot,t)\|_{L^2}
    \|M^2(t-\tau)\|_{L^1}\diff\tau\\
    &\leq
    \left\|S^0\right\|_{L^2}+ \frac{2}{3}
    \left\|M^2\right\|_{C\left([0,T];L^1\right)}
    \int_0^t\|g\|_{L^2}\tau^{-\frac{3}{4}}\diff\tau\\
    &\leq
    \left\|S^0\right\|_{L^2}+ \frac{8}{3}
    \left\|M\right\|^2_{C\left([0,T];L^2\right)}
    \|g\|_{L^2}t^{\frac{1}{4}}\\
    &\leq
    \left\|S^0\right\|_{L^2}+ \frac{8}{3}
    \left\|M\right\|^2_{C\left([0,T];L^2\right)}
    \|g\|_{L^2}T^{\frac{1}{4}}.
\end{align}
Using the fact that 
$\left\|M\right\|_{C\left([0,T];L^2\right)} \leq 
2\left\|S^0\right\|_{L^2},$
and recalling that $\frac{32}{3}\|g\|_{L^2}\left\|S^0\right\|_{L^2}T^\frac{1}{4}<1,$
we can see that
\begin{align}
    \frac{8}{3}\left\|M\right\|^2_{C\left([0,T];L^2\right)}
    \|g\|_{L^2}T^{\frac{1}{4}}
    &\leq \frac{32}{3}\left\|S^0\right\|_{L^2}^2\|g\|_{L^2}T^\frac{1}{4}\\
    &\leq \left\|S^0\right\|_{L^2}.
\end{align}
This implies that 
\begin{equation}
    \|W[M]\|_{C\left([0,T];L^2\right)}\leq 2\left\|S^0\right\|_{L^2}.
\end{equation}
Therefore $W$ is an automorphism on the closed ball
\begin{equation}
    B=\left\{M\in C\left([0,T];L^2_{st}\right):
    \|M\|_{C\left([0,T];L^2\right)}
    \leq 2\left\|S^0\right\|_{L^2}\right\}.
\end{equation}

We will now show that $W$ is a contraction mapping on $B.$
Fix $M,Q\in B.$ Using Young's convolution inequality as above we can compute that for all $0\leq t \leq T$
\begin{align}
    \|W[M](t)-W[Q](t)\|_{L^2}&=\frac{1}{3}
    \left\|P_{st}\int_0^te^{t\Delta}\left(
    (M+Q)(M-Q)+(M-Q)(M+Q)\right)\right\|_{L^2}\\
    &\leq
    \frac{8}{3}\|g\|_{L^2}t^\frac{1}{4}
    \|M+Q\|_{C\left([0,T];L^2\right)}
    \|M-Q\|_{C\left([0,T];L^2\right)}\\
    &\leq
     \frac{8}{3}\|g\|_{L^2}t^\frac{1}{4}
    \left(\|M\|_{C\left([0,T];L^2\right)}
    +\|Q\|_{C\left([0,T];L^2\right)}\right)
    \|M-Q\|_{C\left([0,T];L^2\right)}\\
    &\leq
    \frac{32}{3}\|g\|_{L^2}
    \left\|S^0\right\|_{L^2} T^\frac{1}{4}
    \|M-Q\|_{C\left([0,T];L^2\right)}.
\end{align}
Letting 
\begin{align}
    r&=\frac{32}{3}\|g\|_{L^2}
    \left\|S^0\right\|_{L^2} T^\frac{1}{4}\\
    &<1,
\end{align}
we can find that
\begin{equation}
    \|W[M]-W[Q]\|_{C\left([0,T];L^2\right)} \leq
    r\|M-Q\|_{C\left([0,T];L^2\right)}.
\end{equation}
Note that $B$ is a complete metric space so by the Banach fixed point theorem, we can conclude that there exists a unique
$S\in B\subset C\left([0,T];L^2_{st}\right),$ such that
\begin{equation}
    W[S]=S.
\end{equation}

This implies that there is a unique, mild solution with initial data in $S^0\in L^2_{st}$ locally in time. Note that the higher regularity $S\in C\left((0,T];H^\infty\right)$ is a result of the smoothing due to the heat kernel, but we will not go through the details of that here. This higher regularity follows from a a bootstrapping argument that is essentially the same as the argument in the case of the Navier--Stokes equation given in \cite{KatoFujita}.

\end{proof}

We will now prove a useful proposition giving an identity for the determinant of $3 \times 3$, symmetric, trace free matrices.

\begin{proposition} \label{DetId}
Suppose $M\in \mathbb{S}^{3\times 3}$ is a $3\times 3$ symmetric matrix
such that $\tr(M)=0.$ Then
\begin{equation}
    \tr(M^3)=3 \det(M).
\end{equation}
\end{proposition}
\begin{proof}
Every symmetric matrix is diagonalizable over $\mathbb{R},$ so let $\lambda_1\leq \lambda_2 \leq \lambda_3$ be the eigenvalues of $M.$
Using the trace free condition we can see that 
\begin{align}
\tr(M)&=\lambda_1+\lambda_2+\lambda_3\\ &=0.
\end{align}
Therefore we can compute that
\begin{align}
    \tr\left(M^3\right)&= \lambda_1^3+\lambda_2^3+\lambda_3^3\\
    &= (-\lambda_1-\lambda_2)^3+\lambda_1^3+\lambda_2^3\\
    &=-3\lambda_1^2\lambda_2 -3\lambda_1\lambda_2^2\\
    &=3(-\lambda_1-\lambda_2)\lambda_1\lambda_2\\
    &= 3\lambda_1\lambda_2\lambda_3\\
    &= 3\det(M).
\end{align}
This completes the proof.
\end{proof}

Using this proposition, we will show that the strain self-amplification model equation has the same identity for enstrophy growth as the Navier--Stokes strain equation.

\begin{proposition} \label{ModelEnstrophyGrowth}
Suppose $S \in C\left([0,T_{max});L^2_{st}\right)$ is a mild solution to the strain self-amplification model equation. 
Then for all $0<t<T_{max},$
\begin{align}
    \frac{\diff}{\diff t} \|S(\cdot,t)\|_{L^2}^2&=
    -2\|S\|_{\dot{H}^1}^2- 
    \frac{4}{3}\int\tr\left(S^3\right)\\
    &=
    -2\|S\|_{\dot{H}^1}^2- 4\int\det(S).
\end{align}
\end{proposition}
\begin{proof}
Taking the derivative in time of the $L^2$ norm, we plug into the strain self-amplification model equation \eqref{StrainModel}, 
finding that
\begin{align}
    \frac{\diff}{\diff t} \|S(\cdot,t)\|_{L^2}^2&=
    2\left<\partial_t S,S\right>\\
    &=-2\left<-\Delta S+\frac{2}{3}P_{st}
    \left(S^2\right),S\right>\\
    &=-2\|S\|_{\dot{H}^1}^2
    -\frac{4}{3}\left<P_{st}\left(S^2\right), S\right>\\
    &=-2\|S\|_{\dot{H}^1}^2
    -\frac{4}{3}\left<S^2, S\right>\\
    &=-2\|S\|_{\dot{H}^1}^2
    -\frac{4}{3}\int \tr\left(S^3\right)\\
    &=-2\|S\|_{\dot{H}^1}^2-4\int \det(S),
\end{align}
where we have used the fact that $S\in L^2_{st}$ to drop the projection $P_{st}$, and the fact that $S$ is symmetric to compute the inner product, and finally applied Proposition \ref{DetId}. This completes the proof.
\end{proof}

In fact, the vortex stretching and the integral of the determinant of the strain can be related in a general way as follows. This will be useful in showing the term we dropped in the model equation does not contribute to enstrophy growth.

\begin{proposition} \label{VortexEquivalence}
For all $S\in L^3_{st},$
\begin{equation}
    -4\int\det(S)=\left<S;\omega\otimes\omega\right>,
\end{equation}
where $u=-2\divr(-\Delta)^{-1}S$ 
and $\omega=\nabla\times u.$
In particular this implies that
\begin{equation}
    \left<S,P_{st}\left(\frac{1}{3}S^2
    +\frac{1}{4}\omega\otimes\omega
    \right)\right>=0.
\end{equation}
\end{proposition}

\begin{proof}
The first step of the proof will be to show that
\begin{equation} \label{IntZero}
    \int\tr\left(\left(\nabla u\right)^3\right)=0.
\end{equation}
We will begin by recalling that by definition
\begin{equation}
    S=\nabla_{sym}(-\Delta)^{-1}(-2\divr(S)),
\end{equation}
and so we can see that 
$S=\nabla_{sym}u$.
We may conclude that
\begin{equation}
    \nabla u=-2\nabla \divr (-\Delta)^{-1}S,
\end{equation}
Using the boundedness of the Riesz from 
$L^3 \longrightarrow L^3$,
this implies that $\nabla u\in L^3,$ so the integral in \eqref{IntZero} is absolutely convergent.

Using the divergence free condition we note that
\begin{equation}
    \sum_{i=1}^3\frac{\partial u_i}{\partial x_i}
    = \nabla \cdot u =0.
\end{equation}
Therefore for all $u\in C^\infty_c\left(\mathbb{R}^3;\mathbb{R}^3\right), 
\nabla \cdot u=0,$ we can integrate by parts---without worrying about boundary terms because of compact support---finding that
\begin{align}
    \int\tr\left((\nabla u)^3\right)&=
    \sum_{i,j,k=1}^3\int
    \frac{\partial u_j}{\partial x_i}
    \frac{\partial u_k}{\partial x_j}\frac{\partial u_i}{\partial x_k}\\
    &=
    -\sum_{i,j,k=1}^3\int
    u_j \frac{\partial^2 u_k}{\partial x_i \partial x_j} 
    \frac{\partial u_i}{\partial x_k}\\
    &=
    \sum_{i,j,k=1}^3\int
    u_j \frac{\partial u_k}{\partial x_i} 
    \frac{\partial^2 u_i}{\partial x_j \partial x_k}\\
    &=
    -\sum_{i,j,k=1}^3\int
    \frac{\partial u_j}{\partial x_k}
    \frac{\partial u_k}{\partial x_i}\frac{\partial u_i}{\partial x_j}\\
    &=-\int\tr\left((\nabla u)^3\right)\\
    &=0.
\end{align}
Because $C^\infty_c\left(\mathbb{R}^3;\mathbb{R}^3\right)$ is dense in 
$L^3\left(\mathbb{R}^3;\mathbb{R}^3\right),$ this is sufficient to guarantee that for all $\nabla u\in L^3, \nabla\cdot u=0,$
\begin{equation}
    \int \tr\left(\left(\nabla u\right)^3\right)=0.
\end{equation}

We know that $\nabla u=S+A.$ Using the fact that $S$ is symmetric and $A$ is anti-symmetric, and that all anti-symmetric matrices are trace free, we compute that
    \begin{equation}
\tr\left((\nabla u)^3\right)=\tr\left(S^3\right)+3\tr\left(S A^2\right).
    \end{equation}
Recall from the introduction that 
\begin{equation}
A=\frac{1}{2} \left (
\begin{matrix}
0 & \omega_3 & -\omega_2 \\
-\omega_3 & 0 & \omega_1 \\
\omega_2 & -\omega_1 & 0 \\
\end{matrix}
\right ),
\end{equation}
and we can compute that
\begin{equation}
    A^2=\frac{1}{4}\omega\otimes\omega
    -\frac{1}{4}|\omega|^2 I_3.
\end{equation}
Therefore we find that
\begin{align}
    3\tr\left(S A^2\right)&=
    \frac{3}{4}\left(S:\omega\otimes\omega\right)
    +\frac{3}{4}|\omega|^2\tr(S)\\
    &=\frac{3}{4}\left(S:\omega\otimes\omega\right).
\end{align}
Applying Proposition \ref{DetId}, we find that
\begin{equation}
    \tr\left(S^3\right)=3\det(S).
\end{equation}
Therefore we find that
\begin{equation}
    \tr\left((\nabla u)^3\right)=3\det(S)
    +\frac{3}{4}\left(S:\omega\otimes\omega\right).
\end{equation}
Integrating this equality over $\mathbb{R}^3$ we find
\begin{align}
    \left<S;\omega\otimes\omega\right>+4\int\det(S)
    &=
    \frac{4}{3}\int\left(3\det(S)+\frac{3}{4}
    \left(S:\omega\otimes\omega\right)\right)\\
    &=
    \frac{4}{3}\int\tr\left(\left(\nabla u\right)^3\right)\\
    &=0.
\end{align}
Finally we will compute that
\begin{align}
    \left<P_{st}\left(\frac{1}{3}S^2+\frac{1}{4}
    \omega\otimes\omega\right);S\right>
    &=
    \left<\frac{1}{3}S^2+\frac{1}{4}\omega\otimes\omega;S\right>\\
    &=
    \frac{1}{4}\left<S;\omega\otimes\omega\right>
    +\frac{1}{3}\int\tr\left(S^3\right)\\
    &=
    \frac{1}{4}\left<S;\omega\otimes\omega\right>+\int\det(S)\\
    &=0,
\end{align}
and this completes the proof. The author would like to thank the anonymous referee from an earlier version of \cite{MillerStrain} for this observation.
\end{proof}

Using this result, we will observe that the term we have dropped from the Navier--Stokes strain equation to obtain our strain self-amplification model equation is orthogonal to $S$ with respect to the $L^2$ inner product.

\begin{corollary} \label{OrthogonalPerturbation}
Suppose $S\in H^1_{st},$ with $S=\nabla_{sym}u$ and $\omega=\nabla \times u,$ then 
\begin{equation}
    \left<P_{st}\left((u\cdot\nabla)S+\frac{1}{3}S^2+\frac{1}{4}
    \omega\otimes\omega\right);S\right>=0.
\end{equation}
\end{corollary}

\begin{proof}
We begin by observing that $H^1 \hookrightarrow \dot{H}^\frac{1}{2} \hookrightarrow L^3$, and so clearly $S\in L^3_{st}$. 
Applying Proposition \ref{VortexEquivalence} we can see that
\begin{equation}
    \left<P_{st}\left(\frac{1}{3}S^2+\frac{1}{4}
    \omega\otimes\omega\right);S\right>=0.
\end{equation}
Next we use the divergence free condition, $\nabla \cdot u=0,$ and the fact that we have sufficient regularity to integrate by parts to compute that
\begin{align}
    \left<P_{st}\left((u\cdot\nabla)S\right);S\right>
    &=
    \left<(u\cdot\nabla)S;S\right>\\
    &=
    -\left<S;(u\cdot\nabla)S\right>\\
    &=0.
\end{align}
This completes the proof.
\end{proof}
Note that this means that the term $P_{st}\left((u\cdot\nabla)S
+\frac{1}{3}S^2+\frac{1}{4}\omega\otimes\omega\right),$ does not contribute to enstrophy growth, so when we write the Navier--Stokes strain equation as
\begin{equation}
    \partial_t S-\Delta S+\frac{2}{3}P_{st}\left(S^2\right)
    +P_{st}\left((u\cdot\nabla)S+\frac{1}{3}S^2+\frac{1}{4}
    \omega\otimes\omega\right)=0,
\end{equation}
only the terms $-\Delta S$ and $\frac{2}{3}P_{st}\left(S^2\right)$ contribute to enstrophy growth. This is the justification for studying the dynamics of enstrophy growth using a model equation that drops the term $P_{st}\left((u\cdot\nabla)S+\frac{1}{3}S^2
+\frac{1}{4}\omega\otimes\omega\right),$
retaining only the terms that actually contribute to the growth of enstrophy.

The strain self-amplification model equation, like the Navier--Stokes strain equation, is invariant under the rescaling
\begin{equation}
    S^\lambda(s,t)=\lambda^2S(\lambda x,\lambda^2 t).
\end{equation}
We will now show the existence of global smooth solutions of the strain self-amplification model equation with small initial data in the critical Hilbert space
$\dot{H}^{-\frac{1}{2}}$.
\begin{theorem} \label{SmallData}
Suppose $S^0\in L^2_{st}\cap\dot{H}^{-\frac{1}{2}}_{st}$ and
\begin{equation}
    \left\|S^0\right\|_{\dot{H}^{-\frac{1}{2}}}<\frac{3\sqrt{3}}{4\sqrt{2}}\pi.
\end{equation}
Then there exists a unique, global smooth solution to the strain self-amplification model equation $S\in C\left([0,+\infty);L^2_{st}\right),$ that is $T_{max}=+\infty.$
\end{theorem}
\begin{proof}
We begin by observing there must be a smooth solution 
$S\in C\left((0,T_{max});L^2_{st}\right),$ for some $T_{max}>0.$
We will consider the growth of the $\dot{H}^{-\frac{1}{2}}$ on this time interval.
We will use the fractional Sobolev inequality proven by Lieb \cites{Lieb,LiebLoss}. For all $f \in L^\frac{3}{2}\left(\mathbb{R}^3\right)$
\begin{equation}
    \|f\|_{\dot{H}^{-\frac{1}{2}}}
    \leq \frac{1}{2^\frac{1}{6}\pi^\frac{1}{3}}
    \|f\|_{L^\frac{3}{2}},
\end{equation}
and for all $g\in L^3\left(\mathbb{R}^3\right),$
\begin{equation}
    \|g\|_{L^3}
    \leq \frac{1}{2^\frac{1}{6}\pi^\frac{1}{3}}
    \|g\|_{\dot{H}^\frac{1}{2}}.
\end{equation}
Applying both fractional Sobolev inequalities we find that
\begin{align}
    \frac{\diff}{\diff t}
    \|S(t)\|_{\dot{H}^{-\frac{1}{2}}}^2
    &=
    -2\|S\|_{\dot{H}^\frac{1}{2}}^2-\frac{4}{3}\left<
    (-\Delta)^{-\frac{1}{2}}S,S^2\right>\\
    &\leq
    -2\|S\|_{\dot{H}^\frac{1}{2}}^2+\frac{4}{3}
    \left\|(-\Delta)^{-\frac{1}{2}}S\right\|_{\dot{H}^\frac{1}{2}}
    \left\|S^2\right\|_{\dot{H}^{-\frac{1}{2}}}\\
    &\leq
    -2\|S\|_{\dot{H}^\frac{1}{2}}^2
    +\frac{4}{3}\frac{1}{2^\frac{1}{6}\pi^\frac{1}{3}}
    \left\|S\right\|_{\dot{H}^{-\frac{1}{2}}}
    \left\|S^2\right\|_{L^\frac{3}{2}}\\
    &\leq
    -2\|S\|_{\dot{H}^\frac{1}{2}}^2
    +\frac{4}{3}\frac{1}{2^\frac{1}{6}\pi^\frac{1}{3}}
    \left\|S\right\|_{\dot{H}^{-\frac{1}{2}}}
    \left\|S\right\|_{L^3}^2\\
    &\leq
    -2\|S\|_{\dot{H}^\frac{1}{2}}^2
    +\frac{4}{3}\frac{1}{2^\frac{1}{2}\pi}
    \left\|S\right\|_{\dot{H}^{-\frac{1}{2}}}
    \left\|S\right\|_{\dot{H}^\frac{1}{2}}^2\\
    &\leq
    2\|S\|_{\dot{H}^\frac{1}{2}}^2 \left(
    -1+\frac{\sqrt{2}}{3\pi}\|S\|_{\dot{H}^{-\frac{1}{2}}}\right).
\end{align}
From this bound on the growth of the $\dot{H}^{-\frac{1}{2}}$ norm it is clear that if
\begin{equation}
    \|S(t)\|_{\dot{H}^{-\frac{1}{2}}}<\frac{3\pi}{\sqrt{2}},
\end{equation}
then
\begin{equation}
    \frac{\diff}{\diff t} \|S(t)\|_{\dot{H}^{-\frac{1}{2}}}<0.
\end{equation}
We know that 
\begin{align}
\left\|S^0\right\|_{\dot{H}^{-\frac{1}{2}}}
&<\frac{3\sqrt{3}}{4\sqrt{2}}\pi\\
&<\frac{3\pi}{\sqrt{2}},
\end{align}
so we can conclude that for all $0\leq t <T_{max}$
\begin{equation}
    \|S(t)\|_{\dot{H}^{-\frac{1}{2}}}<\frac{3\sqrt{3}}{4\sqrt{2}}\pi.
\end{equation}
To finish the proof we will need to consider bounds on the enstrophy growth in terms of the $\dot{H}^{-\frac{1}{2}}$ norm.
In addition to the fractional sharp Sobolev inequality, we will also make use of the ordinary sharp Sobolev inequality \cites{Talenti,Sobolev}, which states that for all $f\in L^6\left(\mathbb{R}^3\right),$
\begin{equation}
    \|f\|_{L^6}\leq 
    \frac{1}{\sqrt{3}}\left(\frac{2}{\pi}\right)^\frac{2}{3}
    \|f\|_{\dot{H}^1}.
\end{equation}
Applying the Sobolev inequality, the fractional Sobolev inequality, H\"older's inequality, and the product rule to the identity for enstrophy growth Proposition \ref{ModelEnstrophyGrowth}, we find
\begin{align}
    \frac{\diff}{\diff t} \|S(t)\|_{L^2}^2
    &=
    -2\|S\|_{\dot{H}^1}^2-\frac{4}{3}\left<S,S^2\right>\\
    &\leq
    -2\|S\|_{\dot{H}^1}^2+\frac{4}{3}\|S\|_{\dot{H}^{-\frac{1}{2}}}
    \left\|S^2\right\|_{\dot{H}^{\frac{1}{2}}}\\
    &=
    -2\|S\|_{\dot{H}^1}^2
    +\frac{4}{3}\|S\|_{\dot{H}^{-\frac{1}{2}}}
    \left\|\nabla\left(S^2\right)
    \right\|_{\dot{H}^{-\frac{1}{2}}}\\
    &\leq 
    -2\|S\|_{\dot{H}^1}^2+\frac{4}{3}
    \frac{1}{2^\frac{1}{6}\pi^\frac{1}{3}}
    \|S\|_{\dot{H}^{-\frac{1}{2}}}
    \left\|\nabla\left(S^2\right)\right\|_{L^\frac{3}{2}}\\
    &\leq 
    -2\|S\|_{\dot{H}^1}^2+\frac{4}{3}
    \frac{1}{2^\frac{1}{6}\pi^\frac{1}{3}}
    \|S\|_{\dot{H}^{-\frac{1}{2}}}
    2\|\nabla S\|_{L^2}\|S\|_{L^6}\\
    &\leq 
    -2\|S\|_{\dot{H}^1}^2+\frac{8}{3}
    \frac{1}{2^\frac{1}{6}\pi^\frac{1}{3}}
    \frac{1}{\sqrt{3}}\left(\frac{2}{\pi}\right)^\frac{2}{3}
    \|S\|_{\dot{H}^{-\frac{1}{2}}}
    \|S\|_{\dot{H}^1}^2\\
    &=
    2\|S\|_{\dot{H}^1}^2\left(-1+\frac{4\sqrt{2}}{3\sqrt{3}\pi}
   \|S\|_{\dot{H}^{-\frac{1}{2}}}\right).
\end{align}
We have already shown that for all $0\leq t<T_{max},$
\begin{equation}
    \frac{4\sqrt{2}}{3\sqrt{3}\pi}
   \|S(t)\|_{\dot{H}^{-\frac{1}{2}}}<1,
\end{equation}
so for all $0\leq t<T_{max},$
\begin{equation}
    \frac{\diff}{\diff t} 
    \|S(t)\|_{L^2}^2\leq 0.
\end{equation}
This implies that for all $0\leq t<T_{max},$
\begin{equation}
    \|S(t)\|_{L^2}\leq\left\|S^0\right\|_{L^2}.
\end{equation}
We know from Theorem \ref{ModelMildExistence} that for all $0\leq t<T_{max},$
\begin{equation}
    T_{max}-t>\frac{C}{\|S(t)\|_{L^2}^4}.
\end{equation}
This means that if $T_{max}<+\infty,$ then 
\begin{equation}
    \lim_{t\to T_{max}}\|S(t)\|_{L^2}=+\infty.
\end{equation}
We know that for all $0\leq t<T_{max},
\|S(t)\|_{L^2}\leq \left\|S^0\right\|_{L^2},$ so we can conclude 
that $T_{max}=+\infty.$ This completes the proof.
\end{proof}

\begin{remark}
We will note that the assumption 
$S\in \dot{H}^{-\frac{1}{2}}\cap L^2$ is not actually necessary; it is sufficient to have small initial data in $\dot{H}^{-\frac{1}{2}},$ to guarantee global regularity with no assumption that $S^0\in L^2.$ However, dropping this assumption makes the proof a little more technical, and, more importantly, the whole point of a strain self-amplification model equation is to model enstrophy growth, so if our solution is not in $L^2$ the model does not mean very much.

Likewise, some of the other results in this section are not optimal: for example it should be straightforward to prove the local existence of mild, smooth solutions with initial data in $B^{-2+\frac{3}{p}}_{p,\infty},$ for $2\leq p<+\infty,$ without too much difficulty. Because the strain self-amplification model equation is adapted specifically to study $L^2$ solutions however, getting local existence or small data results down to the largest scale critical spaces is not particularly useful or illuminating.
\end{remark}

We will now prove that because the strain self-amplification model equation has the same identity for enstrophy growth as the Navier--Stokes equation, it also has a regularity criterion on the positive part of the middle eigenvalue of the strain matrix that is precisely the same as the analogous result for the Navier--Stokes equation, Theorem \ref{NSregcrit}.

\begin{theorem} \label{ModelRegCrit}
Suppose $S\in C\left([0,T_{max});L^2_{st}\right)$ is a mild solution to the strain self-amplification model equation. 
Let $\lambda_1(x,t) \leq \lambda_2(x,t) \leq \lambda_3(x,t)$ be the eigenvalues of $S(x,t),$ and 
let $\lambda^+_2(x,t)=\max\left\{0,\lambda_2(x,t)\right\}.$
Then for all $\frac{3}{q}+\frac{2}{p}=2,
\frac{3}{2}<q\leq +\infty,$ there exists $C_q>0$ depending only on $q$ such that for all $0<t<T_{max}$,
\begin{equation} \label{BoundZ}
    \|S(t)\|_{L^2}^2\leq \left\|S^0\right\|_{L^2}^2 \exp\left(C_q\int_0^t
    \left\|\lambda_2^+(\tau)\right\|_{L^q}^p\diff\tau\right).
\end{equation}
In particular, if $T_{max}<+\infty,$ then
\begin{equation}
    \int_0^{T_{max}}\left\|\lambda_2^+(t)\right\|_{L^q}^p\diff t
    =+\infty.
\end{equation}
\end{theorem}

\begin{proof}
We know from Theorem \ref{ModelMildExistence} that if $T_{max}<+\infty,$ then
\begin{equation}
    \lim_{t \to T_{max}}\|S(t)\|_{L^2}^2=+\infty,
\end{equation}
so it suffices to prove the estimate \eqref{BoundZ}.
Because $\tr(S)=0,$ we know that $\lambda_1 \leq 0$ and 
$\lambda_3\geq 0.$ Therefore we know that
\begin{equation}
    -\lambda_1\lambda_3\geq 0.
\end{equation}
We can therefore compute from the identity for enstrophy growth in Proposition \ref{ModelEnstrophyGrowth} that
\begin{align}
    \frac{\diff}{\diff t}
    \|S(t)\|_{L^2}^2
    &=
    -2\|S\|_{\dot{H}^1}^2-4\int \det(S)\\
    &=
    -2\|S\|_{\dot{H}^1}^2+4\int(-\lambda_1 \lambda_3)\lambda_2\\
    &\leq
    -2\|S\|_{\dot{H}^1}^2+4\int(-\lambda_1 \lambda_3)\lambda_2^+\\
    &\leq
    -2\|S\|_{\dot{H}^1}^2+2\int\lambda_2^+ |S|^2\\
    &\leq
    C_q \|\lambda_2^+\|_{L^q}^p \|S\|_{L^2}^2,
\end{align}
after applying H\"older's inequality, the Sobolev inequality, and Young's inequality. This computation is precisely the same as the one done in the proof of the regularity criterion on $\lambda_2^+$ in \cite{MillerStrain}, so we refer the reader there for more detail on these steps.
Applying Gr\"onwall's inequality, we find for all $0<t<T_{max},$
\begin{equation}
    \|S(t)\|_{L^2}^2\leq \left\|S^0\right\|_{L^2}\exp\left(C_q\int_0^t
    \|\lambda_2^+(\tau)\|_{L^q}^p\diff\tau \right).
\end{equation}
This completes the proof.
\end{proof}

This regularity criterion means that there must be unbounded planar stretching in the scale critical $L^p_t L^q_x$ spaces in order for finite-time blowup to occur. The strength of the strain formulation of the Navier--Stokes regularity problem means that not only does the strain self-amplification model equation respect geometric regularity criteria in terms of the strain; it also respects the regularity criterion on two components of the vorticity proven for the full Navier--Stokes equation by Chae and Choe in \cite{ChaeVort}.

\begin{corollary}
Suppose $S\in C\left([0,T_{max});L^2_{st}\right)$ is a mild solution to the strain self-amplification model equation. 
Let $\omega=\nabla \times (-\Delta)^{-1} (-2\divr(S))$ be the vorticity associated with the strain $S$.
Then for all $\frac{3}{q}+\frac{2}{p}=2,
\frac{3}{2}<q<+\infty,$ there exists $C_q>0$ depending only on $q$ such that for all $0<t<T_{max}$,
\begin{equation}
    \|S(t)\|_{L^2}^2\leq \left\|S^0\right\|_{L^2}^2 
    \exp\left(C_q\int_0^t
    \left\|e_3\times\omega(\tau)\right\|_{L^q}^p
    \diff\tau\right).
\end{equation}
In particular, if $T_{max}<+\infty,$ then
\begin{equation}
    \int_0^{T_{max}}\left\|e_3\times\omega(t)
    \right\|_{L^q}^p\diff t
    =+\infty.
\end{equation}
\end{corollary}

\begin{proof}
We know that $\tr(S)=0$ and $\lambda_1\leq \lambda_2\leq \lambda_3$, so $\lambda_2$ is the 
smallest eigenvalue of $S$ in magnitude.
This implies that for all unit vectors $v\in\mathbb{R}^3$,
$|\lambda_2|\leq |Sv|$.
Consequently we can see that for all $x\in\mathbb{R}^3$,
\begin{equation}
    |\lambda_2|\leq |Se_3|.
\end{equation}
Next we observe that 
\begin{align}
    2S e_3&=\nabla u_3+\partial_3 u \\
    e_3\times\omega &= \nabla u_3-\partial_3 u.
\end{align}
We can see that $\nabla u_3$ is a gradient and that
$\nabla \cdot \partial_3 u=0$,
and so using the Helmholtz projections unto the spaces of divergence free vector fields and gradients, we can see that
\begin{align}
    \nabla u_3&=P_{gr}(e_3\times \omega) \\
    \partial_3 u&= -P_{df}(e_3\times \omega).
\end{align}
The boundedness of the Helmholtz decomposition then implies that for all $1<q<+\infty$,
\begin{align}
    \|\lambda_2\|_{L^q}
    &\leq
    \|Se_3\|_{L^q} \\
    &\leq
    \frac{1}{2}\|\nabla u_3\|_{L^q}
    +\frac{1}{2}\|\partial_3 u\|_{L^q}\\
    &\leq
    C_q \|e_3\times\omega\|_{L^q}
\end{align}
The result then follows as an immediate corollary of Theorem \ref{ModelRegCrit}.
\end{proof}

\section{Finite-time blowup for the strain self-amplification model equation}

In this section, we will prove the existence of finite-time blowup for the strain self-amplification model equation. We begin by proving a nonlinear differential inequality giving a lower bound on the rate of enstrophy growth that is sufficient to guarantee finite-time blowup
for some initial data.

\begin{proposition} \label{BlowupRate}
Suppose $S\in C\left([0,T_{max});H^1_{st}\right)$ is a mild solution of the strain self-amplification model equation.
Then for all $0 \leq t<T_{max}$,
\begin{equation}
    \frac{\diff}{\diff t} E(t)
    \geq
    g_0 E(t)^\frac{3}{2},
\end{equation}
where
\begin{equation}
    g_0=\frac{-3\left\|S^0\right\|_{\dot{H}^1}^2
    -4\int\det\left(S^0\right)}
    {\left\|S^0\right\|_{L^2}^3}.
\end{equation}
\end{proposition}

\begin{proof}
We will begin by letting
\begin{align}
    g(t)
    &=
    \frac{-3\|S(\cdot,t)\|_{\dot{H}^1}^2
    -4\int\det(S(\cdot,t))}
    {\|S(\cdot,t)\|_{L^2}^3} \\
    &=
    \frac{-3\|S\|_{\dot{H}^1}^2
    -\frac{4}{3}\int\tr\left(S^3\right)}
    {E(t)^\frac{3}{2}}.
\end{align}
Differentiating $g$, we find that
for all $0<t<T_{max}$,
\begin{align}
    \frac{\diff}{\diff t} g(t) 
    &=
    \frac{6\|-\Delta S+\frac{2}{3}P_{st}
    \left(S^2\right)\|_{L^2}^2}
    {E(t)^\frac{3}{2}}
    -\frac{3}{2}\frac{
    \left(-3\|S\|_{\dot{H}^1}^2
    -\frac{4}{3}\int\tr\left(S^3\right)\right)
    \left(-2\|S\|_{\dot{H}^1}^2
    -\frac{4}{3}\int\tr\left(S^3\right)\right)}
    {E(t)^\frac{5}{2}} \\
    &=
    \frac{6\|-\Delta S+\frac{2}{3}P_{st}
    \left(S^2\right)\|_{L^2}^2}
    {\|S\|_{L^2}^3}
    -\frac{3}{2}\frac{
    \left(-3\|S\|_{\dot{H}^1}^2
    -\frac{4}{3}\int\tr\left(S^3\right)\right)
    \left(-2\|S\|_{\dot{H}^1}^2
    -\frac{4}{3}\int\tr\left(S^3\right)\right)}
    {\|S\|_{L^2}^5} \\
    &\geq
    \frac{6\|-\Delta S+\frac{2}{3}P_{st}
    \left(S^2\right)\|_{L^2}^2}
    {\|S\|_{L^2}^3}
    -\frac{3}{2}\frac{
    \left(-2\|S\|_{\dot{H}^1}^2
    -\frac{4}{3}\int\tr\left(S^3\right)\right)^2}
    {\|S\|_{L^2}^5} \\
    &=
    \frac{6}{\|S\|_{L^2}^5}
    \left(\|S\|_{L^2}^2 
    \left\|-\Delta S+\frac{2}{3}P_{st}
    \left(S^2\right)\right\|_{L^2}^2
    -\left(-\|S\|_{\dot{H}^1}^2
    -\frac{2}{3}\int\tr\left(S^3\right)\right)^2
    \right).
\end{align}
Applying Young's inequality, we find
\begin{align}
    -\|S\|_{\dot{H}^1}^2
    -\frac{2}{3}\int\tr\left(S^3\right)
    &=
    -\left<-\Delta S+\frac{2}{3}P_{st}
    \left(S^2\right),S\right> \\
    &\leq
    \left\|\Delta S+\frac{2}{3}P_{st}\left(S^2\right)
    \right\|_{L^2}\|S\|_{L^2},
\end{align}
and so we can conclude that for all $0<t<T_{max}$,
\begin{equation}
    \frac{\diff}{\diff t} g(t)\geq 0.
\end{equation}
Therefore, we can conclude that for all $0<t<T_{max}$,
\begin{equation}
    g(t)\geq g_0.
\end{equation}
Finally, we observe that for all
$0<t<T_{max}$,
\begin{align}
    \frac{\diff}{\diff t} E(t)
    &=
    -2\|S\|_{\dot{H}^1}^2-4\int\det(S) \\
    & \geq  \label{StepD}
    -3\|S\|_{\dot{H}^1}^2-4\int\det(S) \\
    &=
    g(t)E(t)^\frac{3}{2} \\
    & \geq 
    g_0 E(t)^\frac{3}{2},
\end{align}
and this completes the proof.
\end{proof}

\begin{remark}
Note that as long as $S$ is not the trivial solution---as long as $\|S\|_{\dot{H}^1}^2>0$---then the inequality in \eqref{StepD} is strict and therefore
for all $0<t<T_{max}$,
\begin{equation}
    \frac{\diff}{\diff t} E(t)>g_0 E(t)^\frac{3}{2}.
\end{equation}
\end{remark}

This differential inequality is sufficient to guarantee finite-time blowup for any solution with initial data such that $g_0>0$. We will now prove Theorem \ref{ModelBlowupIntro}, which is restated here for the reader's convenience.

\begin{theorem} \label{ModelBlowup}
    Suppose $S\in C\left([0,T_{max});H^1_{st}\right)$ is a mild solution of the strain self-amplification model equation, such that
\begin{equation}
    -3\left\|S^0\right\|_{\dot{H}^1}^2
    -4\int\det\left(S^0\right)>0.
\end{equation}
Then for all $0 < t<T_{max}$,
\begin{equation} \label{LowerBound}
    E(t)>\frac{E_0}{(1-r_0 t)^2},
\end{equation}
where
\begin{equation}
    r_0=\frac{-3\left\|S^0\right\|_{\dot{H}^1}^2
    -4\int\det\left(S^0\right)}
    {2\left\|S^0\right\|_{L^2}^2}.
\end{equation}
Note in particular that this implies 
\begin{equation}
    T_{max} \leq \frac{2\left\|S^0\right\|_{L^2}^2}
    {-3\left\|S^0\right\|_{L^2}^2
    -4\int\det\left(S^0\right)}.
\end{equation}
Furthermore, for all $\frac{2}{p}+\frac{3}{q}=2,
\frac{3}{2}<q\leq +\infty,$
\begin{equation}
    \int_0^{T_{max}}\|\lambda_2^+(t)\|_{L^q}^p
    \diff t=+\infty.
\end{equation}
\end{theorem}

\begin{proof}
The main argument of the proof will be integrating the differential inequality in Theorem \ref{BlowupRate}.
Applying Proposition \ref{BlowupRate},
we can see that for all $0<t<T_{max}$,
\begin{equation} \label{StepGrowthBound}
    \frac{\diff}{\diff t} E(t)>g_0E(t)^\frac{3}{2},
\end{equation}
where $g_0>0$ by hypothesis.
Applying the chain rule and \eqref{StepGrowthBound},
for all $0<t<T_{max}$,
\begin{align}
    \frac{\diff}{\diff t} \left(-E(t)^{-\frac{1}{2}}\right)
    &=
    \frac{1}{2}E(t)^{-\frac{3}{2}}
    \frac{\diff}{\diff t} E(t) \\
    &>
    \frac{1}{2}g_0,
\end{align}
Integrating this differential inequality, 
we find that for all $0<t<T_{max}$,
\begin{equation}
    E_0^{-\frac{1}{2}}-E(t)^{-\frac{1}{2}}
    >\frac{1}{2}g_0 t.
\end{equation}
This implies that
\begin{equation}
    E(t)^{-\frac{1}{2}}<
    E_0^{-\frac{1}{2}}-\frac{1}{2}g_0t,
\end{equation}
and therefore that
\begin{equation}
    E(t)>\frac{1}{\left(E_0^{-\frac{1}{2}}
    -\frac{1}{2}g_0 t\right)^2}.
\end{equation}
Multiplying the numerator and denominator by $E_0$,
we find that for all $0<t<T_{max}$,
\begin{equation}
    E(t)>\frac{E_0}
    {\left(1-\frac{1}{2}g_0E_0^\frac{1}{2}t\right)^2}.
\end{equation}
It is easy to check that 
\begin{equation}
    r_0=\frac{1}{2}g_0E_0^\frac{1}{2},
\end{equation}
so we have now established that for all $0<t<T_{max}$,
\begin{equation} 
    E(t)>\frac{E_0}{(1-r_0 t)^2}.
\end{equation}
Furthermore, this clearly implies that
\begin{align}
    T_{max}
    &\leq 
    \frac{1}{r_0} \\
    &=
    \frac{2\left\|S^0\right\|_{L^2}^2}
    {-3\left\|S^0\right\|_{L^2}^2
    -4\int\det\left(S^0\right)}.
\end{align}

Finally, applying Theorem \ref{ModelRegCrit},
we conclude that for all $\frac{2}{p}+\frac{3}{q}=2,
\frac{3}{2}<q\leq +\infty,$
\begin{equation}
    \int_0^{T_{max}}\|\lambda_2^+(t)\|_{L^q}^p
    \diff t=+\infty,
\end{equation}
and this completes the proof.
\end{proof}

Next we will show that the set of initial data satisfying the hypothesis of Theorem \ref{ModelBlowup} is nonempty and bounded below in $\dot{H}^{-\frac{1}{2}}$. We will also show that $\lambda_2^+$ is bounded below in $L^\frac{3}{2}$ for all $S$ in this set. First we will need to perform a few calculations related to the determinant of the strain.

\begin{proposition} \label{StrainSign}
There exists $S\in H^1_{st},$ axisymmetric and swirl-free, such that 
\begin{equation}
    -\int_{\mathbb{R}^3}\det(S)>0.
\end{equation}
Note that we will say that $S$ is axisymmetric and swirl-free if $S=\nabla_{sym}u,$ where $u$ is an axisymmetric, swirl-free, divergence free vector field.
\end{proposition}

\begin{proof}
We begin by taking $u\in H^2_{df}$ 
using axisymmetric coordinates, letting
\begin{equation}
    u(x)= (r-2rz^2)\exp\left(-r^2-z^2\right)e_r
    +(-2z+2r^2z)\exp\left(-r^2-z^2\right) e_z.
\end{equation}
We will observe that
\begin{equation}
    u(x)= \left((1-2x_3^2)
    \left(\begin{array}{c}
         x_1  \\
         x_2 \\
         0
    \end{array}\right)
    + \left(-2x_3+2\left(x_1^2+x_2^2\right)x_3\right)
    \left(\begin{array}{c}
         0  \\
         0 \\
         1
    \end{array}\right)
    \right)
    \exp\left(-\left(x_1^2+x_2^2+x_3^2\right)\right),
\end{equation}
and so not only do we have $u\in H^2$, but we have the stronger result that $u$ must be in the Schwartz class of smooth functions, which have, along with all their derivatives, faster than polynomial decay at infinity.
Taking the divergence of $u$ we find that
\begin{align}
    \nabla\cdot u
    &=
    \left(\partial_r+\frac{1}{r}\right)u_r
    +\partial_z u_z \\
    &=
    \Big(\left(2-4z^2-2r^2+4r^2z^2\right)
    +\left(-2+2r^2+4z^2-4r^2 z^2\right) \Big)
    \exp\left(-r^2-z^2\right) \\
    &= 
    0,
\end{align}
as required so that $u\in H^2_{df}$.
Taking the curl of $u$ we will find that 
\begin{align}
    \omega
    &=
    \left(\partial_z u_r-\partial_r u_z\right)
    e_\theta \\
    &=
    \Big(\left(-4rz-2rz+4rz^3\right)
    -(4rz+4rz-4r^3 z)\Big)
    \exp\left(-r^2-z^2\right) e_\theta \\
    &=
    (-14rz+4rz^3+4r^3z) 
    \exp\left(-r^2-z^2\right) e_\theta.
\end{align}

Next we will observe that the gradient can be represented in axisymmetric coordinates as
\begin{equation}
    \nabla=\frac{1}{r}e_\theta \partial_\theta
    +e_r \partial_r +e_z \partial_z.
\end{equation}
Using this representation and recalling that
\begin{equation}
    e_r=\left(\begin{array}{c}
         \cos(\theta)  \\
         \sin(\theta) \\
         0
    \end{array} \right),
\end{equation}
so we can see that
\begin{equation}
    \partial_\theta e_r=e_\theta
\end{equation}
This means we can compute that 
\begin{align}
    \tr\left(S (e_\theta\otimes e_\theta) \right) 
    &=
    \tr\left(\nabla u
    (e_\theta \otimes e_\theta) \right) \\
    &=
    \frac{u_r}{r} \\
    &=
    \left(1-2z^2\right)\exp\left(-r^2-z^2\right)
\end{align}

Applying Proposition \ref{VortexEquivalence}
we find that
\begin{align}
    -\int\det(S)
    &=
    \frac{1}{4}\left<S; 
    \omega\otimes\omega\right> \\
    &=
    \frac{1}{4} \int_{\mathbb{R}^3}
    \tr\left(S (e_\theta\otimes e_\theta) \right)(x)
    |\omega(x)|^2 \diff x \\
    &=
    \frac{1}{4}\int_0^\infty \int_{-\infty}^\infty
    2\pi r\left(1-2z^2\right)
    \left(-14rz+4rz^3+4r^3z\right)^2
    \exp\left(-3r^2-3z^2\right)
    \diff z \diff r \\
    &=
    \pi\int_0^\infty \int_0^\infty
    r\left(1-2z^2\right)
    \left(-14rz+4rz^3+4r^3z\right)^2
    \exp\left(-3r^2-3z^2\right)
    \diff z \diff r \\
    &=
    4\pi\int_0^\infty \int_0^\infty
    r^3 z^2\left(1-2z^2\right)
    \left(-7+2z^2+2r^2\right)^2
    \exp\left(-3r^2-3z^2\right)
    \diff z \diff r,
\end{align}
using the fact that integrand is even in $z$.
Making the substitution, $v=z^2, w=r^2,$ we find that
\begin{align}
    -\int\det(S)
    &=
    \pi\int_0^\infty \int_0^\infty
    w \sqrt{v}\left(1-2v\right)
    \left(-7+2v+2w\right)^2
    \exp\left(-3w-3v\right)
    \diff v \diff w \\
    &=
    \frac{8 \pi^\frac{3}{2}}{81\sqrt{3}}
\end{align}

Therefore we can conclude that there exists $S\in H^1_{st},$ axisymmetric and swirl-free, such that
\begin{equation}
    -\int\det(S)>0,
\end{equation}
and this completes the proof.
\end{proof}

\begin{theorem} \label{BlowupSet}
Let the set of initial data satisfying the hypotheses of Theorem \ref{ModelBlowup}, $\Gamma_{\mathrm{blowup}}\subset H^1_{st},$ be given by
\begin{equation}
    \Gamma_{\mathrm{blowup}}=\left\{ S\in H^1_{st}:
    -3\|S\|_{\dot{H}^1}^2-4\int\det(S)>0\right\}.
\end{equation}
Then $\Gamma_{\mathrm{blowup}}$ is nonempty.
\end{theorem}

\begin{proof}
Take any $S\in H^1_{st}$ such that 
\begin{equation}
    -\int\det(S)>0.
\end{equation}
We know such an $S$ must exist from 
Proposition \ref{StrainSign}.
If we multiply such an $S\in H^1_{st}$ by a sufficiently large constant we will end up with an element of $\Gamma_{\mathrm{blowup}}.$
In particular we compute
\begin{align}
    \lim_{m\to +\infty} \left(-3\|mS\|_{\dot{H}^1}^2
    -4\int\det(mS)\right)
    &= \lim_{m\to +\infty}
    \left( -3m^2 \left(\|S\|_{\dot{H}^1}^2\right)
    +4m^3\left(-\int\det(S)\right)\right)\\
    &=+\infty.
\end{align}
Therefore we may conclude that 
for all $S\in H^1_{st}$ such that 
$-\int\det(S)>0,$
for sufficiently large $m>0, 
mS\in\Gamma_{\mathrm{blowup}}.$
This completes the proof.
\end{proof}

\begin{remark}
Note that the near the origin, the velocity corresponding to finite-time blowup for the strain self-amplification model equation from Proposition \ref{StrainSign} and Theorem \ref{BlowupSet} has a very similar geometric structure to the $C^{1,\alpha}$ finite-time blowup solution to the Euler equation from \cites{ElgindiFiniteTimeEuler,ElgindiGhoulMasmoudi}.
Both involve planar stretching and axial compression near the origin. In particular, approximating the velocity in Proposition \ref{StrainSign} near the origin by the first order Taylor polynomial,
\begin{equation}
    u(x) \approx r e_r -2z e_z.
\end{equation}
\end{remark}

\begin{remark}
The fact the $\Gamma_{\mathrm{blowup}}$ is nonempty means that the condition in Theorem \ref{ModelBlowup} is satisfied for some initial data, and so we can conclude that there must exist solutions of the strain self-amplification model equation that blowup in finite-time. In addition to knowing the $\Gamma_{\mathrm{blowup}}$ is nonempty, we also know that $\Gamma_{\mathrm{blowup}}$ is bounded below in $\dot{H}^{-\frac{1}{2}}$, because Theorem \ref{SmallData} states that there is global regularity for solutions of the strain self-amplification model equation with small initial data in $\dot{H}^{-\frac{1}{2}}$, and Theorem \ref{ModelBlowup} requires that all of the solutions with initial data in $\Gamma_{\mathrm{blowup}}$ must blowup in finite-time. This can also be shown directly by computation using the relevant Sobolev embeddings along with H\"older's inequality.
In addition, we have a lower bound on the amount of planar stretching for $S\in\Gamma_{\mathrm{blowup}}$ in the form of a lower bound on $\lambda_2^+$ in the scale critical Lebesgue space.
\end{remark}

\begin{proposition}
For all $S\in \Gamma_{\mathrm{blowup}}$,
\begin{equation}
    \left\|\lambda_2^+\right\|_{L^\frac{3}{2}}
    >
    \frac{9}{2}\left(\frac{\pi}{2}
    \right)^\frac{4}{3}.
\end{equation}
\end{proposition}

\begin{proof}
We will prove the contrapositive.
Suppose $S\in H^1_{st}$ with
\begin{equation}
    \left\|\lambda_2^+\right\|_{L^\frac{3}{2}}
    \leq 
    \frac{9}{2}\left(\frac{\pi}{2}
    \right)^\frac{4}{3}.
\end{equation}
We will begin by observing that because $\tr(S)=0$,
we have $\lambda_1\leq 0$ and $\lambda_3\geq 0$, because three positive (respectively negative) eigenvalues would violate the trace free condition.
This implies that $-\lambda_1\lambda_3\geq 0$.
Therefore, we can compute that
\begin{align}
    -\det(S)
    &=
    (-\lambda_1\lambda_3)\lambda_2 \\
    &\leq
    (-\lambda_1\lambda_3)\lambda_2^+ \\
    &\leq
    \frac{1}{2}
    \left(\lambda_1^2+\lambda_2^2\right)
    \lambda_2^+ \\
    &\leq
    \frac{1}{2}\lambda_2^+|S|^2.
\end{align}

Applying this estimate, H\"older's inequality, and the Sobolev inequality, we find that
\begin{align}
    -3\|S\|_{\dot{H}^1}^2-4\int\det(S)
    &\leq
    -3\|S\|_{\dot{H}^1}^2
    +2\int\lambda_2^+|S|^2 \\
    &\leq
    -3\|S\|_{\dot{H}^1}^2
    +2\left\|\lambda_2^+\right\|_{L^\frac{3}{2}}
    \|S\|_{L^6}^2 \\
    &\leq
    -3\|S\|_{\dot{H}^1}^2
    +\frac{2}{3}\left(\frac{2}{\pi}
    \right)^\frac{4}{3} 
    \left\|\lambda_2^+\right\|_{L^\frac{3}{2}}
    \|S\|_{\dot{H}^1}^2\\
    &\leq
    0.
\end{align}
Therefore we can see that 
$S \notin \Gamma_{\mathrm{blowup}}$, and this completes the proof.
\end{proof}

\section{A perturbative blowup condition for the full Navier--Stokes equation}

In this section, we will prove a perturbative condition for blowup, and we will also show that this perturbative condition is satisfied at least for short times.
We begin by proving Theorem \ref{PerturbativeBlowupIntro}, which is restated for the reader's convenience.

\begin{theorem} \label{PerturbativeBlowup}
Suppose $u\in C\left([0,T_{max});H^2_{df}\right)$ is a mild solution of the Navier--Stokes equation such that
\begin{equation}
    f_0:=-3\left\|S^0\right\|_{\dot{H}^1}^2
    -4\int\det\left(S^0\right)>0,
\end{equation}
and for all $0 <t<T_{max}$
\begin{equation} \label{PerturbativeCondition}
    \frac{\left\|P_{st}\left((u\cdot\nabla)S+\frac{1}{3}S^2
    +\frac{1}{4}\omega\otimes\omega
    \right)(\cdot,t)\right\|_{L^2}}
    {\left\|\left(-\Delta S+P_{st}\left(
    \frac{1}{2}(u\cdot\nabla)S+\frac{5}{6}S^2
    +\frac{1}{8}\omega\otimes\omega\right)
    \right)(\cdot,t)\right\|_{L^2}}
    \leq
    2.
\end{equation}
Then there is finite-time blowup with
\begin{equation}
    T_{max}
    <
    T_*
    :=
    \frac{-E_0+\sqrt{E_0^2+f_0 K_0}}
    {f_0}.
\end{equation}
\end{theorem}

\begin{proof}
We will begin by letting
\begin{equation}
    f(t)=
    -3\left\|S(\cdot,t)\right\|_{\dot{H}^1}^2
    -4\int\det\left(S\right)(\cdot,t)
\end{equation}
We know that $\tr(S)=0$, and that therefore
$\det(S)=\frac{1}{3}\int\tr\left(S^3\right)$.
Therefore we can see that
\begin{equation}
    f(t)=
    -3\|S\|_{\dot{H}^1}^2
    -\frac{4}{3}\int\tr\left(S^3\right)
\end{equation}
Differentiating $f$, we find that 
for all $0<t<T_{max}$,
\begin{equation}
    \frac{\diff}{\diff t} f(t)
    =
    6\left<\Delta S+\frac{2}{3}P_{st}\left(S^2\right),
    -\Delta S+ P_{st}\left((u\cdot\nabla)S+S^2
    +\frac{1}{4}\omega\otimes\omega\right)\right>
\end{equation}
Observe that for any $M,Q\in L^2$
\begin{equation}
    \left<M,M+Q\right>=
    \left\|M+\frac{1}{2}Q\right\|_{L^2}^2
    -\frac{1}{4}\|Q\|_{L^2}^2,
\end{equation}
and so letting
$M=-\Delta S+\frac{2}{3}P_{st}\left(S^2\right),
Q=P_{st}\left((u\cdot\nabla)S+\frac{1}{3}S^2
+\frac{1}{4}\omega\otimes\omega\right)$,
we find that for all $0<t<T_{max}$,
\begin{multline}
    \frac{\diff}{\diff t} f(t)
    =
    6\left\|-\Delta S+P_{st}\left(
    \frac{1}{2}(u\cdot\nabla)S+\frac{5}{6}S^2
    +\frac{1}{4}\omega\otimes\omega\right)
    \right\|_{L^2}^2 \\
    -\frac{3}{2}\left\|P_{st}\left((u\cdot\nabla)S
    +\frac{1}{3}S^2+\frac{1}{4}\omega\otimes\omega
    \right)\right\|_{L^2}^2.
\end{multline}
Applying the perturbative condition \eqref{PerturbativeCondition}, we find that
for all $0<t<T_{max}$,
\begin{equation}
    \frac{\diff}{\diff t} f(t)\geq 0,
\end{equation}
and therefore, for all $0<t<T_{max}$,
\begin{equation}
    f(t)\geq f_0.
\end{equation}

Using the identity for enstrophy growth we find that
for all $0<t<T_{max}$,
\begin{align}
    \frac{\diff}{\diff t} E(t)
    &=
    -2\|S\|_{\dot{H}^1}^2-4\int\det(S) \\
    &>
    -3\|S\|_{\dot{H}^1}^2-4\int\det(S) \\
    &=
    f(t) \\
    &\geq
    f_0.
\end{align}
Integrating this differential inequality we find that for all $0<t<T_{max}$,
\begin{equation}
    E(t) > E_0+f_0 t,
\end{equation}
and integrating this lower bound for enstrophy growth,
we find that for all $0<t<T_{max}$,
\begin{equation}
    \int_0^t E(\tau)\diff\tau
    > 
    E_0 t+\frac{1}{2}f_0 t^2
\end{equation}
Now suppose towards contradiction that $T_{max}\geq T_*$.
Using the definition
\begin{equation}
    T_*
    =
    \frac{-E_0+\sqrt{E_0^2+f_0 K_0}}
    {f_0},
\end{equation}
we find that
\begin{align}
    \int_0^{T_*}E(t)\diff t
    &>
    E_0 T_*+\frac{1}{2}f_0 T_*^2 \\
    &=
    \frac{1}{2}K_0.
\end{align}
However this contradicts the bound from the energy equality, which requires that
\begin{equation}
    \int_0^{T_*} E(t) \diff t \leq 
    \frac{1}{2}K_0.
\end{equation}
Therefore we may conclude that $T_{max}< T_*$,
and this complete the proof.
\end{proof}

We cannot show that the perturbative condition \eqref{PerturbativeCondition} is satisfied up until $T_*$---if we could this would resolve the Navier--Stokes regularity problem by proving the existence of finite-time blowup. We can however, show that it is satisfied for short times. The first step will be to show that it holds at the level of initial data.

\begin{proposition} \label{StationaryCondition}
There exists $S\in H^2_{st}\cap\dot{H}^{-1}_{st}$,
such that
\begin{equation} \label{condition1}
    -3\|S\|_{\dot{H}^1}^2-4\int\det(S)>0
\end{equation}
and
\begin{equation} \label{condition2}
    \frac{\left\|P_{st}\left((u\cdot\nabla)S
    +\frac{1}{3}S^2
    +\frac{1}{4}\omega\otimes\omega
    \right)\right\|_{L^2}}
    {\left\|-\Delta S+P_{st}\left(
    \frac{1}{2}(u\cdot\nabla)S+\frac{5}{6}S^2
    +\frac{1}{8}\omega\otimes\omega\right)
    \right\|_{L^2}}
    <
    2.
\end{equation}
\end{proposition}

\begin{proof}
Begin by taking $M\in H^2_{st}\cap\dot{H}^{-1}_{st}$
such that
\begin{equation}
    -3\|M\|_{\dot{H}^1}^2-4\int\det(M)>0,
\end{equation}
and $Q\in H^2_{st}\cap\dot{H}^{-1}_{st}$ 
not identically zero.
For all $\lambda>0$, let
\begin{align}
    Q^\lambda(x)&=Q(\lambda x) \\
    S^\lambda &= M+ Q^\lambda.
\end{align}
It is a simple computation to observe that
\begin{align}
    \left\|Q^\lambda\right\|_{\dot{H}^1}
    &= \lambda^{-\frac{1}{2}}
    \left\|Q\right\|_{\dot{H}^1} \\
    \left\|Q^\lambda\right\|_{L^3}
    &= \lambda^{-1}
    \left\|Q\right\|_{L^3}
\end{align}
Therefore we can see that
\begin{equation}
    \lim_{\lambda\to+\infty} Q^\lambda=0
\end{equation}
in both $\dot{H}^1$ and $L^3$.
This implies that
\begin{align}
    \lim_{\lambda\to +\infty}
    -3\left\|S^\lambda\right\|_{\dot{H}^1}^2
    -4\int\det\left(S^\lambda\right)
    &=
    -3\left\|M\right\|_{\dot{H}^1}^2
    -4\int\det\left(M\right) \\
    &>0,
\end{align}
so $S^\lambda$ satisfies \eqref{condition1}
for sufficiently large $\lambda>0$.

Now take 
\begin{align}
    v&=-2\divr(-\Delta)^{-1}M \\
    w&=-2\divr (-\Delta)^{-1}Q.
\end{align}
Note that we then have
\begin{align}
    M&=\nabla_{sym}v \\
    Q&=\nabla_{sym}w.
\end{align}
Likewise we will take
\begin{align}
    a&=\nabla \times v \\
    b&= \nabla \times w.
\end{align}
Finally we will let
\begin{align}
    w^\lambda(x)&=\lambda^{-1}w(\lambda x) \\
    b^\lambda(x)&=b(\lambda x),
\end{align}
noting that this implies that
\begin{align}
    Q^\lambda &= \nabla_{sym}w^\lambda \\
    b^\lambda &= \nabla \times w^\lambda.
\end{align}
Going back to our linear combination, 
we can see that
\begin{align}
u^\lambda =v+w^\lambda \\ 
\omega^\lambda=a+b^\lambda.
\end{align}

Applying the triangle inequality we can see that
\begin{multline}
    \left\|P_{st}\left((u^\lambda\cdot\nabla)S^\lambda
    +\frac{1}{3}\left(S^\lambda\right)^2
    +\frac{1}{4}\omega^\lambda\otimes\omega^\lambda
    \right)\right\|_{L^2}
    \leq
    \left\|P_{st}\left((v\cdot\nabla)M+\frac{1}{3}M^2
    +\frac{1}{4} a\otimes a\right)\right\|_{L^2} 
    \\
    +\left\|P_{st}\left((w^\lambda\cdot\nabla)Q^\lambda
    +\frac{1}{3}\left(Q^\lambda\right)^2
    +\frac{1}{4}b^\lambda\otimes b^\lambda\right)\right\|_{L^2}
    \\
    +\left\|P_{st}\left(
    (w^\lambda\cdot\nabla)M
    +(v\cdot\nabla)Q^\lambda
    +\frac{1}{3}\left(Q^\lambda M+MQ^\lambda\right)
    +\frac{1}{4}\left(b^\lambda\otimes a+a \otimes b^\lambda
    \right)\right)\right\|_{L^2},
\end{multline}
and applying H\"older's inequality and our scaling laws from above, we can conclude
\begin{multline} \label{StepA}
    \left\|P_{st}\left((u^\lambda\cdot\nabla)S^\lambda
    +\frac{1}{3}\left(S^\lambda\right)^2
    +\frac{1}{4}\omega^\lambda\otimes\omega^\lambda
    \right)\right\|_{L^2}
    \leq
    \left\|P_{st}\left((v\cdot\nabla)M+\frac{1}{3}M^2
    +\frac{1}{4} a\otimes a\right)\right\|_{L^2} 
    \\
    +\lambda^{-\frac{3}{2}}
    \left\|P_{st}\left((w\cdot\nabla)Q
    +\frac{1}{3}Q^2+\frac{1}{4}b\otimes b
    \right)\right\|_{L^2} 
    +\lambda^{-1}\|w\|_{L^\infty}\|\nabla M\|_{L^2}
    \\
    +\lambda^{-\frac{1}{2}}\|v\|_{L^\infty}\|\nabla Q\|_{L^2}
    +\frac{2}{3}\lambda^{-\frac{3}{4}}\|M\|_{L^4}\|Q\|_{L^4}
    +\frac{1}{2}\lambda^{-\frac{3}{4}}\|a\|_{L^4}\|b\|_{L^4}.
\end{multline}
Likewise we may compute that
\begin{multline} \label{StepB}
    \left\|-\Delta S^\lambda+P_{st}\left(
    \frac{1}{2}(u^\lambda\cdot\nabla)S^\lambda
    +\frac{5}{6}\left(S^\lambda\right)^2
    +\frac{1}{8}\omega^\lambda\otimes\omega^\lambda
    \right)\right\|_{L^2}
    \geq
    \lambda^\frac{1}{2}\|-\Delta Q\|_{L^2} \\
    -\lambda^{-\frac{3}{2}}
    \left\|P_{st}\left(
    \frac{1}{2}(w\cdot\nabla)Q+\frac{5}{6}Q^2
    +\frac{1}{8}b \otimes b\right)
    \right\|_{L^2} \\
    - \left\|-\Delta M+P_{st}\left(
    \frac{1}{2}(v\cdot\nabla)M+\frac{5}{6}M^2
    +\frac{1}{8}a \otimes a\right)\right\|_{L^2}
    \\
    -\frac{5}{3}\lambda^{-\frac{3}{4}}\|M\|_{L^4}\|Q\|_{L^4}
    -\frac{1}{4}\lambda^{-\frac{3}{4}}\|a\|_{L^4}\|b\|_{L^4}
    \\
    -\frac{1}{2}\lambda^{-\frac{1}{2}} 
    \|v\|_{L^\infty}\|\nabla Q\|_{L^2}
    -\frac{1}{2}\lambda^{-1}\|w\|_{L^\infty}\|\nabla M\|_{L^2}.
\end{multline}
Putting together the inequalities in \eqref{StepA} and \eqref{StepB}, we find that
\begin{equation}
    \lim_{\lambda \to +\infty}
    \frac
    {\left\|P_{st}\left((u^\lambda\cdot\nabla)S^\lambda
    +\frac{1}{3}\left(S^\lambda\right)^2
    +\frac{1}{4}\omega^\lambda\otimes\omega^\lambda
    \right)\right\|_{L^2}}
    {\left\|-\Delta S^\lambda+P_{st}\left(
    \frac{1}{2}(u^\lambda\cdot\nabla)S^\lambda
    +\frac{5}{6}\left(S^\lambda\right)^2
    +\frac{1}{8}\omega^\lambda\otimes\omega^\lambda
    \right)\right\|_{L^2}}=0,
\end{equation}
and so in particular for sufficiently large $\lambda>0$,
\begin{equation}
    \frac
    {\left\|P_{st}\left((u^\lambda\cdot\nabla)S^\lambda
    +\frac{1}{3}\left(S^\lambda\right)^2
    +\frac{1}{4}\omega^\lambda\otimes\omega^\lambda
    \right)\right\|_{L^2}}
    {\left\|-\Delta S^\lambda+P_{st}\left(
    \frac{1}{2}(u^\lambda\cdot\nabla)S^\lambda
    +\frac{5}{6}\left(S^\lambda\right)^2
    +\frac{1}{8}\omega^\lambda\otimes\omega^\lambda
    \right)\right\|_{L^2}}<2.  
\end{equation}
This completes the proof.
\end{proof}

Now that we have established that the perturbative condition \eqref{PerturbativeCondition} can hold for initial data, it is straightforward to show that it can hold for at least short times by continuity. This result is Theorem \ref{PerturbativeShortTimeIntro}, which is restated for the reader's convenience.

\begin{theorem} \label{PerturbativeShortTime}
There exists a mild solution of the 
Navier--Stokes equation 
$u\in C\left([0,T_{max});H^3_{df}\right)$
and $\epsilon>0$, such that
\begin{equation}
    -3\left\|S^0\right\|_{\dot{H}^1}^2
    -4\int\det\left(S^0\right)>0
\end{equation}
and for all $0\leq t<\epsilon$
\begin{equation}
    \frac{\left\|P_{st}\left((u\cdot\nabla)S+\frac{1}{3}S^2
    +\frac{1}{4}\omega\otimes\omega
    \right)(\cdot,t)\right\|_{L^2}}
    {\left\|\left(-\Delta S+P_{st}\left(
    \frac{1}{2}(u\cdot\nabla)S+\frac{5}{6}S^2
    +\frac{1}{8}\omega\otimes\omega\right)
    \right)(\cdot,t)\right\|_{L^2}}
    <
    2.
\end{equation}
\end{theorem}

\begin{proof}
Fix initial data $S^0\in H^2_{st}\cap \dot{H}^{-1}_{st}$ such that
\begin{equation}
    -3\left\|S^0\right\|_{\dot{H}^1}^2
    -4\int\det\left(S^0\right)>0
\end{equation}
and
\begin{equation}
    \frac{\left\|P_{st}\left((u^0\cdot\nabla)S^0+\frac{1}{3}\left(S^0\right)^2
    +\frac{1}{4}\omega^0\otimes\omega^0
    \right)\right\|_{L^2}}
    {\left\|-\Delta S^0+P_{st}\left(
    \frac{1}{2}(u^0\cdot\nabla)S^0
    +\frac{5}{6}\left(S^0\right)^2
    +\frac{1}{8}\omega^0\otimes\omega^0\right)
    \right\|_{L^2}}
    <
    2,
\end{equation}
where $S^0=\nabla_{sym} u^0$ and 
$\omega^0=\nabla\times u^0$.
Note that $u^0\in H^3_{df}$ by definition,
and is given by $u^0=-2\divr(-\Delta)^{-1}S^0$.
Let $u \in C\left([0,T_{max});H^3_{df}\right)$
be the unique mild solution of the Navier--Stokes equation with initial data $u^0$.
Next we will let
\begin{multline}
    c(t)=
    \left\|P_{st}\left((u\cdot\nabla)S+\frac{1}{3}S^2
    +\frac{1}{4}\omega\otimes\omega
    \right)(\cdot,t)\right\|_{L^2} \\
    -2\left\|\left(-\Delta S+P_{st}\left(
    \frac{1}{2}(u\cdot\nabla)S+\frac{5}{6}S^2
    +\frac{1}{8}\omega\otimes\omega\right)
    \right)(\cdot,t)\right\|_{L^2}.
\end{multline}
The fact that $u \in C\left([0,T_{max});H^3_{df}\right)$ immediately implies that
$c \in C\left([0,T_{max})\right)$.
We also know by hypothesis that
\begin{equation}
    c(0)<0,
\end{equation}
so by continuity there must exist $\epsilon>0$
such that for all $0\leq t<\epsilon$,
\begin{equation}
    c(t)<0.
\end{equation}
This completes the proof.
\end{proof}

\begin{remark}
The key to the proof of Proposition \ref{StationaryCondition} and Theorem \ref{PerturbativeShortTime} rests on the fact that we can add a perturbative term which is small in both $\dot{H}^1$ and $L^3$, leaving \eqref{condition1} essentially unaffected, but is very large in $\dot{H}^2$, making the denominator in \eqref{condition2} as large as we like. The key is to add a perturbative term that is supported at very high Fourier modes, but with a scaling chosen so that the perturbation remains small in $\dot{H}^1\cap L^3$.
\end{remark}

\begin{remark}
There are axisymmetric, swirl free initial data that satisfy Theorem \ref{PerturbativeShortTime}. To see this, in the context of Proposition \ref{StationaryCondition}, take $M$ as in Proposition \ref{StrainSign} and $Q$ to be an arbitrary axisymmetric, swirl free strain matrix and the result follows.
In this case, however, we know that the perturbative condition can only be satisfied for short times, because there is global regularity for axisymmetric, swirl-free solutions of the Navier--Stokes equation. There is something in the geometry of axisymmetric swirl-free solutions that, when coupled with the dynamics of the equation, guarantees the perturbative condition will fail after short times.
\end{remark}

\begin{corollary}
Suppose $u\in C\left([0,+\infty);H^2_{df}\right)$ is an axisymmetric, swirl-free, mild solution of the Navier--Stokes equation such that
\begin{equation}
    f_0:=-3\left\|S^0\right\|_{\dot{H}^1}^2
    -4\int\det\left(S^0\right)>0.
\end{equation}
Then there exists 
$0< t < T^* :=
\frac{-E_0+\sqrt{E_0^2+f_0 K_0}}{f_0}$, 
such that
\begin{equation} 
    \frac{\left\|P_{st}\left((u\cdot\nabla)S+\frac{1}{3}S^2
    +\frac{1}{4}\omega\otimes\omega
    \right)(\cdot,t)\right\|_{L^2}}
    {\left\|\left(-\Delta S+P_{st}\left(
    \frac{1}{2}(u\cdot\nabla)S+\frac{5}{6}S^2
    +\frac{1}{8}\omega\otimes\omega\right)
    \right)(\cdot,t)\right\|_{L^2}}
    >
    2.
\end{equation}
\end{corollary}

\begin{proof}
Ladyzhenskaya first proved global regularity for solutions of the Navier--Stokes equation with swirl-free, axisymmetric initial data \cites{LadyzhenskayaAxiSym1,LadyzhenskayaAxiSym2}. The corollary follows immediately from Theorem \ref{PerturbativeBlowup} and the global regularity of axisymmetric, swirl-free solutions of the Navier--Stokes equation, because if the perturbative condition from Theorem \ref{PerturbativeBlowup} was satisfied up until $T^*$, then there must be finite-time blowup, which we know cannot occur.
\end{proof}

\section*{Acknowledgements}
This publication was supported in part by the Fields Institute for Research in the Mathematical Sciences while the author was in residence during the Fall 2020 semester. Its contents are solely the responsibility of the author and do not necessarily represent the official views of the Fields Institute.
This material is based upon work supported by the National Science Foundation under Grant No. DMS-1440140 while the author participated in a program that was hosted by the Mathematical Sciences Research Institute in Berkeley, California, during the Spring 2021 semester.
The material in this publication is also based on research conducted by the author while a PhD student at the University of Toronto and a postdoctoral fellow at McMaster University.

\bibliographystyle{plain}
\bibliography{bib}

\end{document}